\documentclass[reqno]{amsart}
\usepackage[english]{babel}
\usepackage[T1]{fontenc}
\usepackage{mathtools,amssymb,verbatim,wrapfig}
\usepackage{physics}
\usepackage{tikz-cd}
\usetikzlibrary{babel}
\usepackage{bm}
\usepackage{csquotes}
\usepackage{float}
\usepackage{enumerate}
\usepackage{xfrac}
\usepackage{mathrsfs}
\usepackage{etoolbox}
\usepackage{verbatim}
\usepackage[colorlinks]{hyperref}
\usepackage{caption,cite,xcolor}
\usepackage[normalem]{ulem}

\AtBeginDocument{%
\hypersetup{
citecolor= [RGB]{114 8 8},
linkcolor= [RGB]{114 8 8},
}}

\newtheorem{theorem}{Theorem}[section]

\newtheorem{prop}[theorem]{Proposition}
\newtheorem{cor}[theorem]{Corollary}
\newtheorem{conjecture}[theorem]{Conjecture}
\theoremstyle{definition}
\newtheorem{definition}[theorem]{Definition}
\newtheorem{example}[theorem]{Example}
\theoremstyle{remark}
\newtheorem*{remark}{Remark}
\newtheorem*{notation}{Notation}

\numberwithin{equation}{section}

\DeclareMathOperator{\im}{Im}

\DeclareMathOperator{\Real}{Re}
\DeclareMathOperator{\id}{id}

\DeclareMathOperator{\spann}{span}

\newcommand{\innerprod}[1]{\left\langle #1 \right\rangle}

\begin{document}

\author{Per \AA hag}
\address{Department of Mathematics and Mathematical Statistics\\ Ume\aa \ University\\ SE-901 87 Ume\aa \\ Sweden}
\email{per.ahag@umu.se}
\author{Rafa\l\ Czy{\.z}}
\address{Institute of Mathematics \\ Faculty of Mathematics and Computer Science \\ Jagiellonian University\\ \L ojasiewicza 6\\ 30-348 Krak\'ow\\ Poland}
\email{Rafal.Czyz@im.uj.edu.pl}
\author{H\aa kan Samuelsson Kalm}
\address{Mathematical Sciences \\ Chalmers University of Technology and University of Gothenburg \\
SE-412 96 \\Gothenburg \\ Sweden}
\email{hasam@chalmers.se}
\author{Aron Persson}
\address{Department of Mathematics, Uppsala University, PO Box 480, S-751 06 Uppsala, Sweden}
\email{aron.persson@math.uu.se}
\title[Manifold-like polyfolds as differential geometrical objects]{On manifold-like polyfolds as differential geometrical objects with applications in complex geometry}

\begin{abstract}
We argue for more widespread use of manifold-like polyfolds (M-polyfolds) as differential geometric objects. M-polyfolds possess a distinct advantage over differentiable manifolds, enabling a smooth and local change of dimension. To establish their utility, we introduce tensors and prove the existence of Riemannian metrics, symplectic structures, and almost complex structures within the M-polyfold framework. Drawing inspiration from a series of highly acclaimed papers by L\'{a}szl\'{o} Lempert, we are laying the foundation for advancing geometry and function theory in complex M-polyfolds.
\end{abstract}

\maketitle

\section{Introduction}

Expanding on Gromov's seminal work on pseudo-holomorphic curves~\cite{Gromov1985} and the subsequent decade and a half of innovative contributions from the mathematical community (see e.g.~\cite{AbbondandoloSchlenk, Floer86, Floer89, Fukaya, FukayaOno, Fukaya2020}), Eliashberg, Givental, and Hofer introduced symplectic field theory at the turn of the millennium~\cite{EliashbergGiventalHofer}. This marked the beginning of an ambitious pursuit to create a unified, abstract framework capable of encompassing diverse theories such as Gromov-Witten Theory, Floer Theory, Contact Homology, and the more general Symplectic Field Theory. In their landmark publications~\cite{hofer2007general1,hofer2009general2,hofer2009general3}, Hofer, Wysocki, and Zehnder ingeniously developed a fitting framework centered around ambient spaces known as Manifold-like polyfolds and their generalizations, polyfolds. For a contemporary and in-depth introduction to Polyfold Theory, we recommend the monograph~\cite{hofer2017polyfold}.

Manifold-like polyfolds (M-polyfolds) share the same local model philosophy as classical differentiable manifolds. However, two distinct differences set them apart. First, the concept of differentiability is generalized. Inspired by the time-shift map  (cf.\ Example~\ref{ex:scsmoothmap} below), Hofer, Wysocki, and Zehnder revisited Krein and his students' work on scale Banach spaces and integrated the crucial insight of requiring compactness of the scale embeddings (Definition~\ref{def:scBanach}). This led to the development of an appropriate notion of scale-differentiability, or sc-differentiability for short, (see Definition~\ref{def:scdiff}). Subsequently, influenced by Cartan's last theorem~\cite{Cartan1986}, Hofer, Wysocki, and Zehnder established that the charts for M-polyfolds should be 
bijections onto sc-smooth retracts instead of open subsets of 
Banach spaces (Definition~\ref{def:scsmoothretract}). This distinction endows M-polyfolds with a unique advantage over differentiable manifolds, enabling a smooth and local change of dimension, as illustrated in Figure~\ref{fig:Mpolyfold}.

The primary objective of this paper is to argue for the more widespread adop-
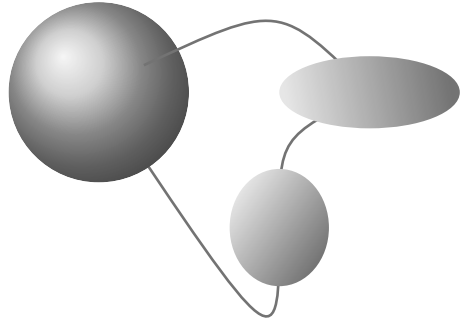
\begin{wrapfigure}{r}{6cm}

 \begin{tikzpicture}[scale=1.2]
    \draw[line width=1pt,color=gray!90!black] (3,0) .. controls(5,-3) .. (5,-1.5);
    \shade[shading=ball, ball color=gray!50] (3,0) circle (1);
    \draw[line width=1pt,color=gray!90!black] (3.5,0.3) .. controls(5,1) .. (6,0);
    \draw[line width=1pt,color=gray!90!black] (5,-1.5) .. controls(5,-0.5) .. (6,0);
    \shade[left color=gray!20,right color=gray!90!black,shading angle={100}] (6,0) ellipse (1 and 0.4);
    \shade[top color=gray!20,bottom color=gray!90!black,shading angle={55}] (5,-1.5) ellipse (0.55 and 0.65);
\end{tikzpicture}
    \caption{An M-polyfold in $\mathbb{R}^3$.}
    \label{fig:Mpolyfold}
\end{wrapfigure}
tion of M-polyfolds as differential geometric objects rather than merely as ambient spaces as initially intended. However, since the polyfold theory is still an emerging field, numerous open-ended questions remain. For instance, the general existence of an sc-smooth partition of unity in M-polyfolds is currently unknown. In this paper, we first introduce the notion of tensors and show that a strong Riemannian metric exists on M-polyfolds of the Hilbert type. We then concentrate on the fundamental class $\Gamma$ (p.~\pageref{ClassGamma}). Our main result is that  even-dimensional M-polyfolds of class $\Gamma$ have an integrable, almost complex structure $J$ and
a compatible Riemannian metric $g$ such that the associated fundamental 2-form $\omega(\cdot,\cdot):=g(J\cdot,\cdot)$ is 
a symplectic form; see Theorem~\ref{thm:almostcomplexmpoly}, Theorem~\ref{thm:nonclosedsymplectic}, and Corollary~\ref{cor:complex}.

\bigskip

In light of our commitment to broadening the applications of M-polyfolds, we shift our focus toward complex geometry. This venture is inspired by L\'{a}szl\'{o}  Lempert's groundbreaking work on complex Banach manifolds~\cite{lempert1998dolbeault, lempert1999dolbeault, lempert2000dolbeault}, which stands as our guiding star. 

\begin{wrapfigure}{l}{6cm}
    \begin{tikzpicture}[scale=1]
	 \fill[color=gray!20] (-1.5,1.5) -- (3.5,1.5) -- (3,2.5) -- (-1,2.5);
    \draw (-1,2.5) -- (3,2.5);
    \shade[left color=gray!20, right color=gray!60] (1,0.05) -- (0,0.05) arc (180:270:1 and 0.55);
    \shade[left color=gray!60, right color=gray!20] (0.95,0.05) -- (2,0.05) arc (360:270:1 and 0.55);
    \draw (0,0) arc (180:360:1 and 0.5);
    \shade[left color=gray!20, right color=gray!60] (0,0) rectangle (1,4);
    \shade[left color=gray!60, right color=gray!20] (0.95,0) rectangle (2,4);
    \filldraw[color=gray!20] (1,4) ellipse (1 and 0.5);
    \draw[color=black!100] (1,4) ellipse (1 and 0.5);
    \draw (0,0) -- (0,4);
    \draw (2,0) -- (2,4) node[anchor=west]{$C$};
    \fill[color=gray!20] (-1.5,1.5) -- (3.5,1.5) -- (4,0.5) -- (-2,0.5);
    \draw (-2,0.5) -- (4,0.5);
    \draw (-2,0.5) -- node[anchor=west]{$P$} (-1,2.5);
    \draw (4,0.5) -- (3,2.5);
    \shade[left color=gray!20, right color=gray!60] (1,1.5) -- (0,1.5) arc (180:270:1 and 0.5);
    \shade[left color=gray!60, right color=gray!20] (0.95,1.5) -- (2,1.5) arc (360:270:1 and 0.5);
    \draw (0,1.5) arc(180:360:1 and 0.5);
    \end{tikzpicture}
\caption{A projected image of the M-polyfold $X=C\cup P$ in $\mathbb{R}^3$.}\label{fig:complexC2Mpolyfold}
\end{wrapfigure}
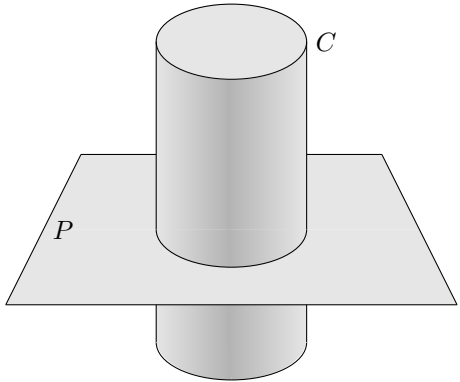

Complex geometry has been a fertile and influential field of research throughout the 20th century. During the sixties and seventies, complex geometry underwent generalization to the infinite-dimensional setting, akin to many other areas in geometry. Within this context, this paper aims to introduce complex M-polyfolds, a concept that extends the ideas of infinite-dimensional complex geometry and has potential applications in various fields of mathematics.

However, as developed in~\cite{lempert1998dolbeault, lempert1999dolbeault, lempert2000dolbeault}, certain concepts do not extend from finite-dimensional to infinite-dimensional settings, even in the relatively tranquil setting of Banach manifolds.

For instance, as outlined in~\cite{patyi2000overline}, the Newlander-Nirenberg theorem does not extend to almost complex Banach manifolds, and there is generally no isomorphism between Dolbeault cohomology and sheaf cohomology. 
However, Dolbeault isomorphism theorems have been established for complex Banach manifolds, although these results are less general than in the finite-dimensional case; see~\cite{patyisimon2009,simon2009}.

Traditionally, a complex manifold is defined as locally biholomorphic to some open subset of $\mathbb{C}^n$. Since this definition is not directly applicable to M-polyfolds, we adapt the concept of complex Banach manifolds employed in~\cite{lempert1998dolbeault} and define a complex M-polyfold as an integrable, almost complex M-polyfold
(see Definition~\ref{def:integrable}  for the notion of integrable). After introducing and examining the fundamental properties of sc-holomorphic functions in complex M-polyfolds using the $\bar{\partial}$-operator, we culminate our study by constructing the following global complex M-polyfold: Let $C=\{(z,w)\in \mathbb{C}^2: |z|<1\}$ be the open solid unit cylinder in $\mathbb{C}^2$, and $P=\{(z,w)\in \mathbb{C}^2: w=0\}$ the complex plane embedded in $\mathbb{C}^2$. The set $X=C\cup P$, along with the complex structure inherited from $\mathbb{C}^2$, constitutes a complex M-polyfold (see Figure~\ref{fig:complexC2Mpolyfold}). As a result, any open subset of $X$ is also a complex M-polyfold.

\section{A primer on M-polyfolds}\label{sec:mpolyprimer}

This section provides a concise overview of M-polyfolds, a crucial concept for understanding the rest of this paper. M-polyfolds are developed in the real setting, where all vector spaces and related concepts are over $\mathbb{R}$.

M-polyfolds can be better understood by drawing a parallel with smooth manifolds and Banach manifolds. A smooth manifold is a topological space locally modeled on Euclidean spaces, formed by seamlessly gluing together subsets of $\mathbb{R}^n$. Banach manifolds extend this notion by employing subsets of Banach spaces instead of $\mathbb{R}^n$. Analogously, M-polyfolds are constructed by coherently joining certain images of Banach space subsets while incorporating an additional structure known as the scales of the Banach space. The scales enable a more intricate analysis of the spaces involved. The gluing procedure in M-polyfolds hinges on the application of maps, referred to as retractions, which exhibit smoothness with respect to the scale structure of the Banach space. This specific type of smoothness ensures that the resulting M-polyfold possesses the desirable properties.

To facilitate the presentation, this section is divided into two parts. Firstly, we introduce the concepts of scale topology and scale Banach spaces, which are prerequisites for understanding M-polyfolds. Secondly, we present the key ideas of M-polyfolds and their applications. For more in-depth information on M-polyfolds, we refer the reader to~\cite{hofer2017polyfold}.

\subsection{Scale topology and scale Banach spaces}

To introduce the concepts of scale topology and scale Banach spaces, we first define the scale structure.

\begin{definition}
\label{def:scBanach}
Let $E$ be a Banach space, and let $\{E_m\}_{m\geq 0}$ be a decreasing sequence of Banach spaces with $E=E_0$. Then we shall say that $\{E_m\}_{m\geq 0}$ is a
\emph{scale structure}, or \emph{sc-structure} for short, on $E$ if:
\begin{enumerate}[$i)$]\itemsep2mm
    \item the inclusion operator $\iota: E_{m+1}\hookrightarrow E_{m}$ is compact, and
    \item the set
     \[
     E_\infty := \bigcap_{i=0}^{\infty}E_i\subseteq E_m
    \]
    is dense for all $m\geq0$.
\end{enumerate}
A Banach space $E$ together with an sc-structure $\{E_m\}_{m\geq0}$ is called an \emph{sc-Banach space}. Let $\norm{\cdot}_m$ denote the norm of the Banach space $E_m$. The points in $E_\infty$ are called  \emph{smooth points} and points in $E_m$ are called  \emph{points of regularity} $m$.
\end{definition}

 The only dense vector subspace of $\mathbb{R}^n$ is $\mathbb{R}^n$ itself. As a consequence, the unique sc-structure on $\mathbb{R}^n$ is the \emph{constant sc-structure}
\begin{equation}
\mathbb{R}^n = E_0 = E_1 = \cdots = E_\infty.
\end{equation}
Hence, finite-dimensional sc-Banach spaces are identical to their classical counterparts. In contrast, consider an infinite-dimensional Banach space $E$. In this case, the constant scales do not form an sc-structure on $E$. This is attributed to the fact that infinite-dimensional Banach spaces are not locally compact,
which implies that the inclusion map $\iota: E\rightarrow E$ cannot be compact.

To better understand why points in $E_m$ are referred to as points of regularity $m$, consider the following example:

\begin{example}
\label{ex:scBanachspace}
Consider the space
\[
\begin{aligned}
E=\mathcal{C}^0(\mathbb{S}^1)
\end{aligned}
\]
i.e., the space of continuous functions $f:\mathbb{S}^1\rightarrow \mathbb{R}$ with the $\sup$-norm, where $\mathbb{S}^1=\mathbb{R}/\mathbb{Z}$. It is a well-known fact that $\mathcal{C}^m(\mathbb{S}^1)$, the $m$-times continuously differentiable functions $f:\mathbb{S}^1\rightarrow \mathbb{R}$, is a Banach space for all $m\in \mathbb{N}$ with the norm
\[
\norm{f}_m = \sup \abs{f}+  \sup \abs{Df} + \dots + \sup \abs{D^m f}.
\]
Therefore, $E_m = \mathcal{C}^{m}(\mathbb{S}^1)$ defines a sequence $\{E_m\}$ of Banach subspaces of $E$. Thanks to the Arzel\`a-Ascoli theorem we have that $E$ together with the collection $\{E_m\}_{m=0}^\infty$ is an sc-structure on $E$.

\qed
\end{example}

\noindent \emph{Notation.}

\begin{itemize}\itemsep2mm

\item[$i)$] Suppose that $E$ and $F$ are sc-Banach spaces. The direct sum  $E\oplus F$ is the direct sum of the Banach spaces $E$ and $F$ equipped with the induced sc-structure $\{E_m\oplus F_m\}_{m=0}^\infty$.

\item[$ii)$]  For an sc-Banach space $E$, a subset $A\subseteq E$ inherits a \textit{filtration} $\{A_m\}_{m=0}^\infty$, $A_m= A\cap E_m$. We will let $A^k$ be the set $A_k$ together with the induced filtration $\{A_{k+m}\}_{m=0}^\infty$. Every subset of an sc-Banach space will be assumed to have the induced filtration.

\end{itemize}

We focus on sc-Banach spaces and examine continuous linear maps that preserve their added sc-structure.

\begin{definition}
\label{def:scaleoperator}
Let $E,F$ be sc-Banach spaces. A linear operator $L:E\rightarrow F$ is called an \emph{sc-operator} if $L(E_m)\subseteq F_m$, and the restriction $L\mid_{E_m}:E_m \rightarrow F_m$ is continuous for all $m\geq 0$. Furthermore, an sc-operator $L:E\rightarrow F$ is said to be an \emph{sc-isomorphism} if it is bijective and its inverse is also an sc-operator.
\end{definition}

In specific applications of M-polyfold theory, M-polyfolds with boundary and corners are required. Thus, following the construction of manifolds with boundaries and corners, the concept of partial quadrants in sc-Banach spaces becomes necessary.

\begin{definition}
Let $E$ and $W$ be sc-Banach spaces, and let $C\subseteq E$ be a closed convex subset. If there exists an sc-isomorphism $L:E\rightarrow \mathbb{R}^n \oplus W$ with $L(C)=[0,\infty)^n \oplus W$, then  $C$ is called a \emph{partial quadrant} in $E$.
\end{definition}

\begin{definition}
\label{def:openlocalmpolyfoldmodel}
Let $E$ be an sc-Banach space, let $C\subseteq E$ be a partial quadrant in $E$, and let $U \subseteq C$ be a relatively open set. Then, the tuple   $(U,C,E)$ is called an \emph{open tuple}.
\end{definition}

We will now discuss scale calculus. This is a notion of
continuity, differentiability, and smoothness that respects the sc-structure.

\begin{definition}
\label{def:scalecontinuous}
Let $(U,C,E)$, $(V,D,F)$ be open tuples. A map $f:U\rightarrow V$ is called \emph{sc}$^0$ (or of \emph{class sc}$^0$, or \emph{sc-continuous}), if $f(U_m)\subseteq V_m$ and the maps $f\mid_{U_m}:U_m\rightarrow V_m$ are continuous for all $m\geq 0$.
\end{definition}

To define differentiability, we shall first define tangent bundles of open tuples. These are analogous to tangent bundles of open subsets of Banach spaces.

\begin{definition}
The \emph{tangent bundle} of the open tuple $(U,C,E)$ is defined as the tuple
\[
T(U,C,E) = (TU,TC,TE)
\]
where
\[
TU=U^1\oplus E,\qquad TC=C^1\oplus E, \qquad TE=E^1\oplus E.
\]
\end{definition}

Observe that $TC$ is a partial quadrant in $E^1 \oplus E$. Additionally,  $TU \subset TC$ is relatively open, resulting in that $(TU, TC, TE)$ is an open tuple .

Furthermore, the tangent bundle is defined on the points of regularity $1$ for the tuple $(U, C, E)$. According to Definition~\ref{def:scBanach}, this implies that the tangent bundle is defined on a dense subset of $E$.

\begin{definition}
\label{def:scdiff}
Let $(U,C,E)$, $(V,D,F)$ be open tuples. A map $f:U\rightarrow V$ is called 
\emph{sc}$^1$-\emph{differentiable} (or of \emph{class sc}$^1$) if the following conditions hold:
\begin{enumerate}[$i)$] \itemsep2mm
    \item $f$ is sc$^0$;
    \item for every $x\in U_1$ there exists a bounded linear operator $A:E\rightarrow F$ such that, for all $h\in E_1$ satisfying $x+h\in U_1$, it holds that
    \[
    \lim_{\norm{h}_1\to 0} \frac{\norm{f(x+h)-f(x)- A(h)}_0}{\norm{h}_1} = 0.
    \]
    We denote the linear operator $A$ satisfying the above by $Df(x)$;
    \item the \emph{tangent map} $Tf:TU \rightarrow TV$ defined for all $x\in U_1$ and all $h\in E$ by
    \[
    Tf(x,h) := (f(x),Df(x)h)
    \]
    is sc$^0$ .
\end{enumerate}
For $k\geq2$, a function $f:U\rightarrow V$ is inductively defined to be \emph{sc}$^{k}$, if $f$ is sc$^{k-1}$ and the map
 \[
 T^{k-1}f:T^{k-1}U =T(T(\underset{k-1}{\dots} T(U))) \rightarrow T^{k-1}V =T(T(\underset{k-1}{\dots} T(V)))
 \]
 is sc$^1$.
Furthermore, $f$ is called \emph{sc}$^\infty$ (or \emph{sc-smooth}) if it is $sc^k$ for all $k \geq 0$.
\end{definition}

It is essential to note that $sc^k$-differentiability is defined in points of regularity $k$, that is points in $U_k\subset C_k\subset  E_k$. In finite-dimensional sc-Banach spaces, the scale-calculus reduces to the classical calculus due to the fact that
the sole sc-structure on $\mathbb{R}^n$ is the constant one. 

The concept of scale-continuity is a more restrictive notion than that of classical
continuity. On the other hand, since $\|h\|_0/\|h\|_1$ is bounded it follows from the above definition that sc$^0$-maps that are Fr\'echet differentiable are also sc$^1$. Thus, being sc$^1$ is less restrictive than being Fr\'echet differentiable if it is given that the function is sc$^0$.
Consequently, the theory of M-polyfolds does not comprehensively encompass the theory of Banach manifolds and vice versa.

The following is a fundamental example of an
sc-smooth map; see, e.g., \cite[Example~4.15]{fabert2016polyfolds}.
\begin{example}
\label{ex:scsmoothmap}
Let $\tau : \mathbb{R}\oplus \mathcal{C}^0(\mathbb{S}^1)\rightarrow \mathcal{C}^0(\mathbb{S}^1)$ be defined by $(s,\gamma) \mapsto \gamma(\cdot + s)$;
recall that $\mathbb{S}^1=\mathbb{R}/\mathbb{Z}$.
For any $\gamma_0,\gamma\in\mathcal{C}^0(\mathbb{S}^1)$ we have
\begin{equation}\label{korvar}
\tau(s_0+s,\gamma_0+\gamma)-\tau(s_0,\gamma_0)=
\gamma_0(\cdot+s_0+s)-\gamma_0(\cdot+s_0) + \gamma(\cdot+s_0+s)
\end{equation}
and it is not hard to see that the restriction of $\tau$ to
$\mathbb{R}\oplus \mathcal{C}^k(\mathbb{S}^1)$ is a continuous mapping
$\mathbb{R}\oplus \mathcal{C}^k(\mathbb{S}^1)\to\mathcal{C}^k(\mathbb{S}^1)$.
Thus $\tau$ is $sc^0$. If $\gamma_0,\gamma\in\mathcal{C}^1(\mathbb{S}^1)$,
then in view of \eqref{korvar},
$$
\tau(s_0+s,\gamma_0+\gamma)-\tau(s_0,\gamma_0)=
\gamma'_0(\cdot +s_0) s  + \gamma(\cdot +s_0)+ o(|s|+|\gamma|_{\mathcal{C}^1}).
$$
It follows that $ii)$ in Definition~\ref{def:scdiff} holds with $D\tau(s_0,\gamma_0)$
being the mapping
$$
(s,\gamma)\mapsto \gamma'_0(\cdot +s_0) s  + \gamma(\cdot +s_0).
$$
This mapping clearly extends to $\mathbb{R}\oplus \mathcal{C}^0(\mathbb{S}^1)$
and one can check that $iii)$ in Definition~\ref{def:scdiff} holds.

\qed
\end{example}

 In the previous example one runs into problems if one tries to define the classical Fr\'echet differential of $\tau$ at a point $(s_0,\gamma_0)$ when $\gamma_0$ is not
$\mathcal{C}^1(\mathbb{S}^1)$ since then $\gamma'_0$ need not even be defined.

\subsection{M-polyfolds}\label{ssecMpoly} From here on, the second half of the local theory shall be summarized. As expected, an sc-smooth version of the chain rule exists.

\begin{theorem}[{\hspace{1sp}\cite[Theorem~1.3.1]{hofer2017polyfold}}]
\label{thm:chainrulescBanach}
Let $(U,C,E)$, $(V,D,F)$, and $(W,Q,G)$ be open tuples. Furthermore, suppose that $f:U\rightarrow V$ and $g:V\rightarrow W$ are sc$^1$-maps. Then, the function $g\circ f$ is an sc$^1$-map and
\[
T(g\circ f) = Tg\circ Tf.
\]
\end{theorem}

One defines sc-smooth diffeomorphisms between open tuples in the following natural way.

\begin{definition}
Let $(U,C,E)$, $(V,D,F)$ be open tuples and let $f:U\rightarrow V$ be a bijection. If $f$ and $f^{-1}$ both are sc$^{k}$, then we call $f$ an \emph{sc$^k$-diffeomorphism}. If $f$ and $f^{-1}$ both are sc-smooth, then we call $f$ an \emph{sc-smooth diffeomorphism}.
\end{definition}

We will now introduce the local models of M-polyfolds. In the initial stages of the theory's development (see \cite{hofer2007general1,hofer2009general2,hofer2009general3}), splicing cores (see  \cite[Definition 3.2]{hofer2007general1}) were utilized as local models. However, based on the more contemporary and comprehensive notion proposed in~\cite{hofer2017polyfold,hofer2010sc, hofer2017polyfoldslecturesongeometry}, we will use sc-smooth retracts as local models. For a comparison between these two models, see \cite[Chapter 5]{fabert2016polyfolds}.

\begin{definition}
\label{def:scsmoothretraction}
Let $(U,C,E)$ be an open tuple. An sc-smooth map $r:U\rightarrow U$ is called an \emph{sc-smooth retraction} if 
\[
r\circ r = r.
\]
\end{definition}

An sc-smooth retraction generalizes a classically differentiable retraction in two ways. Firstly, differentiability is extended to sc-smoothness. Secondly, sc-smooth retractions are defined on open subsets of partial quadrants, not only on the total space as in~\cite{nadler1967differentiable} for classical retractions.

As one might expect, the image of an sc-smooth retraction is called an sc-smooth retract. However, one still keeps track of the partial quadrant and original sc-Banach space to account for the boundary behavior.

\begin{definition}
\label{def:scsmoothretract}
Let $(U,C,E)$ be an open tuple, and let $O\subseteq U$ be a subset. The tuple $(O,C,E)$ is an \emph{sc-smooth retract} if there exists an sc-smooth retraction $r:U\rightarrow U$ such that $r(U)=O$.
\end{definition}

Some results and definitions regarding sc-smooth retracts are needed before discussing the upcoming global spaces.

\begin{definition}
Let $(O,C,E)$ be an sc-smooth retract, and let $r:U\rightarrow U$, $U\subseteq C$ relatively open, be an sc-smooth retraction such that $r(U)=O$. We define the tangent space at a point $x\in O_1$ as the subspace
\[
T_x O = \im{Dr(x)} \subseteq E_0.
\]
Moreover, the \emph{tangent bundle} of $(O,C,E)$ is defined as
\[
T(O,C,E) = (TO,TC,TE),
\]
where
\[
TO=  Tr(TU), \qquad TC=C^1\oplus E, \qquad TE=E^1\oplus E.
\]
\end{definition}

The definition above is invariant under a choice of sc-smooth retraction and is thus well-posed thanks to \cite[Proposition 2.1.11]{hofer2017polyfold}. The subspace
$\im{Dr(x)}$ for $x\in U_1$  is closed in $E_0$ and thus itself is a Banach space.
We endow it with the sc-structure induced by the inclusion $\im{Dr(x)}\subset E$. Notice that $(TO,TC,TE)$ is an
sc-smooth retract of the open tuple $(TU,TC,TE)$.

\begin{definition}
\label{def:scsmooth}
Let $(U,C,E)$, $(V,D,F)$ be open tuples, let $(O,C,E)$, $(P,D,F)$ be sc-smooth retracts with $O\subset U$ and $P\subset V$. Furthermore, suppose that $r:U\rightarrow U$ is an sc-smooth retraction corresponding to $(O,C,E)$. A mapping $f:O\rightarrow P$ is called \emph{sc-smooth}, if
\[
f\circ r: U\rightarrow F
\]
is sc-smooth. Moreover, the \emph{tangent map} $Tf:TO \rightarrow TP$ is defined as
\[
Tf:= T(f\circ r)\mid_{TO}.
\]
\end{definition}

According to \cite[Proposition 2.1.13]{hofer2017polyfold}, Definition~\ref{def:scsmooth} is well-posed.

\begin{prop}[{\hspace{1sp}\cite[Proposition~2.3.1]{hofer2017polyfold}}]
\label{prop:propofopensubsetsofretracts}
Let $(O,C,E)$ be an sc-smooth retract. Then the following statements are true:
\begin{enumerate}[$i)$] \itemsep2mm
    \item if $O'\subseteq O$ is relatively open, then $(O',C,E)$ is an sc-smooth retract;
    \item let $W\subseteq O$ be a relatively open set, and let $s:W\rightarrow W$ be an sc-smooth map such that $s\circ s =s$, then for $O'=s(W)$, $(O',C,E)$ is an sc-smooth retract.
\end{enumerate}
\end{prop}

As indicated above, sc-smooth retracts are used as charts.

\begin{definition}
Let $X$ be a set, let $V\subseteq X$ be a subset, let $(O,C,E)$ be an sc-smooth retract, and let $\phi:V\rightarrow O$ be a bijection. The tuple $(V,\phi,(O,C,E))$ is called a \emph{chart} on $X$.
\end{definition}

It may be advantageous not to have a topology on $X$ a priori but instead let a maximal atlas of sc-smooth retracts induce a topology on $X$. To do that, we need the following definition of an sc-smooth atlas.

\begin{definition}
\label{def:scsmoothatlas}
Let $X$ be a set and let $\{(V_\alpha,\phi_\alpha,(O_\alpha,C_\alpha,E_\alpha))\}_{\alpha\in \mathcal{A}}$ be a collection of charts on $X$. The collection $\{(V_\alpha,\phi_\alpha,(O_\alpha,C_\alpha,E_\alpha))\}_{\alpha\in \mathcal{A}}$ is called an \emph{sc-smooth atlas} on $X$ if
\begin{enumerate}[$i)$] \itemsep2mm
    \item $\bigcup_{\alpha\in \mathcal{A}} V_\alpha = X$;
    \item the sets $\phi_\alpha(V_\alpha\cap V_\beta)$ are relatively open in $O_\alpha$ for all $\alpha,\beta\in\mathcal{A}$;
    \item for every two charts $(V_\alpha,\phi_\alpha,(O_\alpha,C_\alpha,E_\alpha))$ and $(V_\beta,\phi_\beta,(O_\beta,C_\beta,E_\beta))$ such that $V_\alpha \cap V_\beta \neq \emptyset$, the transition maps
    \[
    \phi_\beta\circ\phi_\alpha^{-1}\mid_{\phi_\alpha(V_\alpha\cap V_\beta)}: \phi_\alpha(V_\alpha\cap V_\beta) \rightarrow \phi_\beta(V_\alpha\cap V_\beta)
    \]
    are sc-smooth diffeomorphisms.
\end{enumerate}
\end{definition}

Since $\phi_\alpha(V_\alpha\cap V_\beta)$ are assumed relatively open,
it follows from Proposition~\ref{prop:propofopensubsetsofretracts}
that they are sc-smooth retracts. It thus makes sense to require the
transition functions to be sc-smooth diffeomorphisms.

It should be of no surprise that from Definition~\ref{def:scsmoothatlas}, one can define equivalence classes of sc-smooth atlases. One can use similar arguments as in the classical differential geometric setting to show that the following relation indeed is an equivalence relation; this hinges, though, crucially on Proposition~\ref{prop:propofopensubsetsofretracts} and Theorem~\ref{thm:chainrulescBanach}.

\begin{definition}
Two sc-smooth atlases are 
\emph{equivalent} if the union of the two atlases also is an sc-smooth atlas. The \emph{maximal} sc-smooth atlas for an sc-smooth atlas
\[
\{(V_\alpha,\phi_\alpha,(O_\alpha,C_\alpha,E_\alpha))\}_{\alpha\in \mathcal{A}}
\]
is the union of all equivalent sc-smooth atlases to $\{(V_\alpha,\phi_\alpha,(O_\alpha,C_\alpha,E_\alpha))\}_{\alpha\in \mathcal{A}}$.
Such a maximal sc-smooth atlas shall be referred to as an \emph{M-polyfold structure}.
\end{definition}

 It is now possible to define the \emph{topology induced by an atlas}, $\mathcal{A}$, on a
set $X$. This topology is generated by the collection of chart domains of the maximal sc-smooth atlas generated by $\mathcal{A}$. Analogous to classical manifolds, we assume that this topology should be paracompact and Hausdorff.

\begin{definition}
\label{def:mpolyfold}
Let $X$ be a set together with an M-polyfold structure. If the topology induced by the M-polyfold structure is paracompact and Hausdorff, then $X$ is called an \emph{M-polyfold}.
\end{definition}

In the early stages of the development of the M-polyfold theory \cite{hofer2007general1,hofer2009general2,hofer2009general3}, the paracompact axiom was replaced by second countability. Though the second countability assumption was deemed overly restrictive for applications, and as a result,  \cite{fabert2016polyfolds,hofer2017polyfold} used the paracompactness axiom instead. It is worth noting that M-polyfolds are defined as metrizable spaces in, for example, \cite[Definition 1.15]{hofer2010integrationonzerosets}. The Smirnov metrization theorem establishes the equivalence of these two definitions. For further discussion on the topology assumptions of M-polyfolds, see \cite[Remark 5.2]{fabert2016polyfolds}.

On top of the M-polyfold structure, i.e.\ sc-smooth structure, M-polyfolds inherit \emph{sc-structures} from the sc-structures of the images of its charts. Let $p\in X$ be a point and let $(V,\phi, (O,C,E))$ be a chart around $p$. Then, $p$ is said to be of scale $m$ if $\phi(p)\in O_m$. Using sc-continuity of the transition functions, it is straightforward to verify that this definition is chart-independent. By denoting $X_m$ as the set of all points in $X$ of scale $m$, we obtain a decreasing sequence:
\[
X=X_0 \supseteq X_1 \supseteq \dots \supseteq X_\infty = \bigcap_{i=0}^{\infty} X_i\, .
\]
Furthermore, $(V\cap X_m,\phi\mid_{V\cap X_m}, (O^m,C^m,E^m))$ serve as charts on $X_m$, making $X_m$ an M-polyfold as well.

Defining sc-smoothness for maps on M-polyfolds is also essential. Analogous to classical differential geometry, this is done by composing such maps with charts.

\begin{definition}
\label{def:scsmoothbetweenMpoly}
Let $X$, $Y$ be M-polyfolds and let $f:X\rightarrow Y$ be a mapping. Suppose that $p\in X$ is a point and that $(V,\phi,(O,C,E))$ is a chart around $p$. Furthermore, let $f(p) =q\in Y$, and let $(U,\psi,(P,D,F))$ be a chart around $q$. The map $f$ is \emph{sc-smooth} at $p$ if the map
\[
\psi \circ f \circ \phi^{-1}:O\rightarrow P
\]
is an sc-smooth map between the sc-smooth retracts $O$ and $P$ as in Definition~\ref{def:scsmooth}.
\end{definition}

Following \cite{fabert2016polyfolds,hofer2007general1,hofer2017polyfold}, the tangent space of an M-polyfold at a point is defined using equivalence classes of charts as follows.
Let $X$ be an M-polyfold, let $p\in X_1$ be a point, and let
\[
\mathcal{M}=\{( U_\alpha,\phi_\alpha,(O_\alpha,C_\alpha,E_\alpha))\}_{\alpha \in \mathcal{A}}
\]
be the maximal atlas that defines the M-polyfold structure for $X$. Let 
$\mathcal{A}_p$ be the set of $\alpha$ such that $p\in U_\alpha$ and let
\[
G_p :=\bigcup_{\alpha\in\mathcal{A}_p} \{p\}\times T_{\phi_\alpha (p)} O_\alpha\times \mathcal{M}.
\]
Two elements $(p,v,( U,\phi,(O,C,E))), (p,w,( V,\psi,(P,D,F)))\in G_p$ are said to be \emph{equivalent}, if
\[
v = D\left(\phi \circ \psi^{-1}\right) (\psi(p)) w.
\]
This is indeed an equivalence relation, that we denote by $\sim$. 

\begin{definition}\label{def:Tangent}
The \emph{tangent space of $X$ at the point $p$} is 
\[
T_pX = G_p/\sim.
\]
A \emph{tangent vector at $p$} is the equivalence class of a tuple
$
(p,v,(U,\phi, (O,C,E))).
$
\end{definition}

Using classical arguments, one has that the tangent space at each point is a Banach space. If $r:U\rightarrow U$ is an sc-smooth retraction corresponding to $(O,C,E)$, then $T_pX$ is isomorphic to $\im(Dr(\phi(p)))$.

Next, we give the definition of the dimension of an M-polyfold. It is crucial to observe that although an sc-Banach space $E$ might be infinite-dimensional, an sc-smooth retract $(O,C,E)$ can be finite dimensional. As a result, it is inappropriate to rely on the dimension of $E$ when defining the dimension of an M-polyfold. Instead, the dimension of the tangent space is employed.

\begin{definition}
\label{def:dimofmpolyfold}
Let $X$ be an M-polyfold. A point $p\in X_1$ is of \emph{dimension} $n\in [0,\infty]$ if $T_p X$ is an $n$-dimensional vector space. Furthermore, $X$ is called \emph{finite-dimensional} if
\[
\sup_{p\in X_1} \left(\dim T_pX \right) < \infty,
\]
and $X$ is said to be \emph{even-dimensional} (resp.\ \emph{odd-dimensional})
if for every $p\in X_1$, $\dim (T_p X) =\infty$ or $\dim(T_p X)=2k$
(resp.\ $\dim(T_p X) = 2k+1$), were  $k\in \mathbb{N}$.
\end{definition}

\begin{remark}
In standard differential geometry the tangent space at a point can be defined in several equivalent ways; either as an equivalence class of charts, as an equivalence class of curves, or as derivations of smooth functions. For finite dimensional M-polyfolds these potentially different definitions are again equivalent. However, for infinite dimensional M-polyfolds they are all inequivalent. 
\end{remark}

We now focus on the \emph{tangent bundle} of an M-polyfold. Since the tangent space is only defined on the first scale of the M-polyfold $X$, the tangent bundle must be constructed solely over $X_1$, or on a subset thereof.Consequently, the tangent bundle is given by:
\[
TX := \bigsqcup_{p\in X_1} T_p X.
\]
As proved in \cite[Proposition 2.3.15]{hofer2017polyfold}, the tangent bundle is an M-polyfold.
Let $(V,\phi,(O,C,E))$ be a chart on $X$. If $p\in V$ and $(p,v,(V,\phi,(O,C,E)))$ is
a representative of a tangent vector $v_p\in T_pX$, then we define $T_p \phi:T_pX \rightarrow T_{\phi(p)}O$ by simply setting
\[
T_p \phi(v_p): = (\phi(p),v).
\]
This is indeed well-defined. We get an induced map $T\phi\colon TV\to TO$.
Such maps give an sc-smooth atlas on $TX$.

A sc$^k$-vector field on an M-polyfold is an sc$^k$ map $\mu:X^1\rightarrow TX$, denoted $\mu\in \mathfrak{X}^k(X)$, that satisfies:
\[
\pi\circ \mu = \id_{X_1}
\]
where $\pi:TX\rightarrow X^1$ is the projection of each fibre to the base point. The space of sc-smooth vector fields is simply denoted by $\mathfrak{X}(X)$. The Lie bracket of two vector fields 
is defined as follows; the full rigor behind it can be seen in \cite[Chapter 4.2]{hofer2010integrationonzerosets}.

\begin{definition}
Let $X$ be an M-polyfold, and let $\mu,\nu : X^1 \rightarrow TX$ be two sc$^k$ vector fields. Then we define the map $[\cdot,\cdot]: \mathfrak{X}^k(X) \times \mathfrak{X}^k(X)\rightarrow \mathfrak{X}^{k-1}(X^1)$ by
\begin{equation}\label{lie}
[\mu,\nu](p) = D\mu(p) \nu(p) - D\nu(p) \mu(p),
\end{equation}
where $p\in X_2$ and $k\geq 2$.
\end{definition}

Observe that the Lie bracket raises the scale of the base M-polyfold as the Lie bracket
of $\mu,\nu \in \mathfrak{X}(X)$ is an sc-smooth mapping
$
[\mu,\nu] : X^2 \rightarrow TX^1.
$

The Whitney sum of $k$ copies of $TX$ is defined as the fiber-wise direct sum
\[
\bigoplus^k_{i=1} TX := \bigsqcup_{p\in X_1} \bigoplus^k_{i=1} T_pX.
\]
By convention the empty direct sum is  
\begin{equation}\label{convention1}
\bigoplus_{1}^{0}TX = X_1.
\end{equation}
In a similar way as for $TX$, one gets an sc-smooth atlas on
$\oplus_1^kTX$.

\begin{definition}\label{def: differential between Mpolyfolds}
Let $X,Y$ be M-polyfolds and let $f:X\rightarrow Y$ be an sc$^1$ map. Then, the map $Tf:TX\rightarrow TY$ is the map such that the diagram

\medskip

\[
\begin{tikzcd}[sep=12mm]
TV \arrow[rr, "Tf\mid_{TV}"] \arrow[d,"T\phi"] &  & TU\arrow[d,"T\psi"]\\
TO \arrow[rr,"T(\psi\circ f\mid_{V} \circ \phi^{-1})"] & & TP
\end{tikzcd}
\]

\medskip

\noindent commutes for all charts $(V,\phi,(O,C,E))$, $(U,\psi,(P,D,F))$ in $X$ and $Y$, respectively, with $f(V)\subset U$.
\end{definition}

For this paper, the notion of sc-smooth differential forms is essential. Initially, it was~\cite{hofer2010integrationonzerosets} that pioneered the theory of sc-differential forms, a natural extension from Banach manifolds to M-polyfolds.

\begin{definition}
\label{def:scdifform}
Let $X$ be an M-polyfold, and let
\[
\omega : \bigoplus_{i=1}^k TX \rightarrow \mathbb{R}
\]
be an sc-smooth map such that
$\omega_p \colon T_pX\times\cdots\times T_pX  \rightarrow \mathbb{R}$ is $k$-multilinear and alternating for all $p\in X_1$. Then, for $k\geq 1$,  $\omega$ is called an \emph{sc-differential $k$-form} on $X$. The set of all sc-differential $k$-forms on $X$ is denoted $\Omega^k(X)$. We let $\Omega^0(X)$ be the space of sc-smooth functions on $X$.
\end{definition}

If 
\[
\omega : \bigoplus_{j=1}^k TX \;\longrightarrow\; \mathbb{R}
\]
is fiberwise multilinear and, for every choice of local vector fields $\mu_1, \dots, \mu_k$, the map
\[
\omega(\mu_1,\ldots,\mu_k) : X_1 \;\longrightarrow\; \mathbb{R}
\]
is sc-smooth, then $\omega$ itself is sc-smooth. Consequently, such an $\omega$ is  an \emph{sc-differential $k$-form}.

The scales of M-polyfolds become particularly relevant when examining the theory of sc-differential forms. It is important to note that, through the inclusions of M-polyfolds $\iota:X^{i+1} \hookrightarrow X^{i}$, sc-differential forms on $X^i$ can be naturally pulled back to sc-differential forms on $X^{i+1}$. As a result, it is well-posed to consider the following directed system:
\begin{equation}\label{sladdat}
\Omega^k(X)\rightarrow \Omega^k(X^1) \rightarrow \dots \rightarrow \Omega^k(X^i) \rightarrow \Omega^k(X^{i+1})\rightarrow \dots
\end{equation}
The direct limit of this directed system is denoted as $\Omega_\infty^k (X)$. The direct sum over k-differential forms results in a graded algebra, which, in accordance with the literature, is denoted as:
\[
\Omega_\infty ^* (X) = \bigoplus_{k=0}^\infty \Omega^k_\infty(X).
\]

Next, we define the exterior derivative of sc-differential forms. 

\begin{definition}
\label{def:extdifferential}
Let $X$ be an M-polyfold. The \emph{exterior derivative}
\[
d: \Omega^k(X^i) \rightarrow \Omega^{k+1}(X^{i+1}),
\]
where $i\geq0$, is defined by for each $\omega\in \Omega^k(X^i)$ letting
\begin{multline*}
d\omega(\mu_0,\mu_1,\dots,\mu_k) = \sum_{j=0}^k (-1)^j D(\omega(\mu_0,\dots,\hat{\mu}_j,\dots, \mu_k))\mu_j\\
+ \sum_{j<l} (-1)^{j+l} \omega([\mu_j,\mu_l],\mu_0,\dots,\hat{\mu}_j,\dots,\hat{\mu}_l,\dots,\mu_k),
\end{multline*}
where $\mu_0,\mu_1,\dots,\mu_k\in \mathfrak{X}(X)$ and where $\hat{\mu}_i$ denotes that the $i$:th element is omitted.
\end{definition}

The following diagram 
\[
\begin{tikzcd}
\Omega^k(X^i) \arrow[r, "d"] \arrow[d, hook] & \Omega^{k+1}(X^{i+1}) \arrow[d, hook]\\
\Omega^k(X^{i+1}) \arrow[r, "d"]& \Omega^{k+1}(X^{i+2})
\end{tikzcd},
\]
where the vertical mappings are the pullback mappings \eqref{sladdat}, commutes for all $i\in \mathbb{N}$. It is thus possible to define the induced exterior derivative
\[
d:\Omega_\infty ^* (X) \rightarrow \Omega_\infty ^* (X).
\]
This map satisfies the classical identity $d^2 =0$. Therefore all exact differential forms are closed and therefore as seen in \cite[Definition 4.4.6]{hofer2017polyfold} there exists a de Rham cohomology on M-polyfolds.

\smallskip

We conclude this section by presenting the construction of the M-polyfold family $\Gamma$, which is crucial in this paper. The initial concept of this family can be attributed to \cite[Example 1.22]{hofer2010sc}, with a minor modification later introduced in \cite[Example 5.8]{fabert2016polyfolds} (see also~\cite{hofer2017polyfold}). In this paper, we provide a straightforward generalization of this concept.\label{ClassGamma}

\medskip

Let $\{\delta_m\}$, $\delta_m\in \mathbb{R}$, be a strictly increasing sequence with $\delta_0=0$. Furthermore, let $\beta:\mathbb{R}\rightarrow [0,1]$ be a smooth odd function such that $\beta(0)=0$ and $\beta(s)=1$ whenever $s\geq 1$. The Hilbert space $L^2(\mathbb{R})$ of real-valued functions, is equipped with the following sc-structure of weighted Sobolev spaces
\begin{equation}
\label{eq:constructionscbanachspace}
\begin{aligned}
E_m= W^{m,\delta_m} &:= \Bigg\{f \in L^2(\mathbb{R}): \\
&\int_\mathbb{R} \left(|f(s)|^2 + |D^1f(s)|^2 +\dots + |D^m f(s)|^2 \right) e^{\delta_m s \beta(s)} ds < \infty\Bigg\}.
\end{aligned}
\end{equation}
That this is indeed an sc-structure is shown in \cite[Lemma 4.10]{fabert2016polyfolds}. The main argument for compactness follows from Sobolev's embedding theorem for weighted Sobolev spaces (see, e.g.~\cite{gol2009weighted}). One defines an sc-smooth retraction on $L^2(\mathbb{R})$ as follows.
Begin by choosing $k$ compactly supported, smooth, mutually orthonormal functions
\[
\gamma_1,\dots,\gamma_k \in L^2(\mathbb{R}).
\]

For each $t\in\mathbb{R}$ let $\pi^k_t\colon L^2(\mathbb{R})\rightarrow L^2(\mathbb{R})$, be defined by $\pi_t^k=0$ if $t\leq 0$, and
\[
\pi^k_t(f) = \left\langle f,\gamma_1(\cdot + e^{\frac{1}{t}}) \right\rangle_{L^2} \gamma_1(\cdot + e^{\frac{1}{t}}) + \dots +\left\langle f,\gamma_k(\cdot + e^{\frac{1}{t}}) \right\rangle_{L^2} \gamma_k(\cdot + e^{\frac{1}{t}})
\]
if $t>0$. That is, $\pi^k_t$ is orthogonal projection on the subspace
generated by the translated functions
$\gamma_1(\cdot + e^{\frac{1}{t}}),\ldots,\gamma_k(\cdot + e^{\frac{1}{t}})$.
By slightly modifying the proof of \cite[Lemma 1.23]{hofer2010sc} it follows that  $\pi^k:\mathbb{R}\oplus L^2(\mathbb{R}) \rightarrow L^2(\mathbb{R})$, $\pi^k(t,x):= \pi^k_t(x)$, is sc-smooth. The map $r:\mathbb{R}^n\oplus \mathbb{R} \oplus L^2(\mathbb{R})\rightarrow \mathbb{R}^n\oplus \mathbb{R} \oplus L^2(\mathbb{R})$ defined by
\begin{equation}\label{lunch}
r(x_1,\dots,x_n,t,f)= (x_1,\dots,x_n,t,\pi_t^k(f))
\end{equation}
is therefore sc-smooth. Since $\pi_t^k$ is a projection it follows that
$r=r\circ r$ and so $r$ is an sc-smooth retraction. Let
\begin{equation}\label{gammank}
\Gamma_n^k = r(\mathbb{R}^n \oplus \mathbb{R}\oplus L^2(\mathbb{R})).
\end{equation}

The identity map defined on $\Gamma_n^k$ generates an sc-smooth structure on $\Gamma_n^k$, which tautologically generates a topology on $\Gamma_n^k$ that is paracompact and Hausdorff. In fact, $\Gamma_n^k$ is homeomorphic to the set
\[
\left(\mathbb{R}^n \times \mathbb{R}_{\leq 0} \times\{0\} \right) \bigcup \left( \mathbb{R}^n\times \mathbb{R}_{>0} \times \mathbb{R}^k\right)
\]
with the subspace topology inherited from the standard topology.  We notice that
$\Gamma^k_n$ is even-dimensional if and only if $n$ is an odd integer and 
$k$ is an even integer.
The class $\Gamma$ is the class of all M-polyfolds $\Gamma_n^k$.

Figure~\ref{fig:1dimplus2dim} illustrates $\Gamma^1_0$, which is homeomorphic to an open half-plane together with a line.

\begin{remark}\label{rem: Gamma}
    Note that the topology of $\Gamma^k_n$ is independent of choice of $\gamma_j$, $\delta_m$,  and $\beta$. Moreover, the sc-smooth structure is independent of $\gamma_j$. However, it is not clear to us whether the sc-smooth structure of $\Gamma^k_n$ is independent of the choice of $\delta_m$ and $\beta$. 
\end{remark}

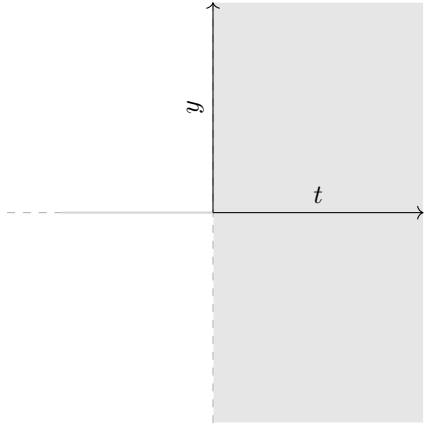
\begin{figure}[H]
\centering
\begin{tikzpicture}[scale=0.8]
\draw[color=white,fill=gray!20] (0,-3.5) rectangle (3.5,3.5);
\draw[dashed,color=gray!60] (0,-3.5)--(0,3.5);
\draw[color=gray!30, line width=0.9pt] (0,0)--(-2.5,0);
\draw[dashed, color=gray!60] (-2.5,0) -- (-3.5,0);
\draw[->] (0,0) -- (3.5,0) node[pos=.5,sloped,above] {$t$};
\draw[->] (0,0) -- (0,3.5) node[pos=.5,sloped,above] {$y$};
\end{tikzpicture}
    \caption{An illustration of $\Gamma_0^1$.}
    \label{fig:1dimplus2dim}
\end{figure}

A common way to use $\Gamma_n^k$ to construct M-polyfolds is shown in the following example.

\begin{example}
Proposition~\ref{prop:propofopensubsetsofretracts} implies that open subsets of sc-smooth retracts are sc-smooth retracts. Therefore it is now possible to model M-polyfolds using open subsets of different $\Gamma^k_n$, potentially using different $k$'s and $n$'s. An example of this type of M-polyfold is shown in Figure~\ref{fig:Mpolyfold} on page~\pageref{fig:Mpolyfold}.
\qed
\end{example}

\section{Riemannian Metrics, Symplectic and Almost Complex Structures} \label{sec:tensor}

In this section, we begin by introducing tensors and Riemannian metrics. Following this, we establish the existence of strong Riemannian metrics on certain M-polyfolds
(of the Hilbert type) before discussing symplectic and almost complex structures.

We need to be somewhat careful defining $(r,s)$-tensors on M-polyfolds. Note that the dual of an sc-Banach space does not admit a natural sc-structure and therefore, our definition of $(r,s)$-tensors does not involve any notion of cotangent bundle. Instead, we shall make use of tensor products of the tangent bundle as building blocks to define tensor fields. First, we check that the tensor product of two sc-Banach spaces has a natural sc-structure.

Let $E,F$ be Banach spaces. Let $E\otimes_{alg} F$ denote the \textit{algebraic tensor product} in the sense of vector spaces. On $E\otimes_{alg} F$ we impose the injective tensor norm, see \cite{Ryan2002introduction}, for $\omega\in E\otimes_{alg} F$
\[
\norm{\omega}_\epsilon := \sup_{\lambda\in E^*,\mu\in F^*}\left\{ (\lambda\otimes \mu)(\omega) : \norm{\lambda}_{E^*}\leq 1,\norm{\mu}_{F^*}\leq 1\right\},
\]
where we have naturally defined the tensor product of functionals in $E^*\otimes_{alg} F^*$. By $E\otimes_\epsilon F$ we denote the completion of $E\otimes_{alg} F$ under the injective norm. 

\begin{theorem}
\label{thm:tensorscbanach}
Let $E$ and $F$ be sc-Banach spaces. Then, the Banach space $E\otimes_\epsilon F$ admits the sc-structure
\[
\{E_n \otimes_\epsilon F_n\}_{n=0}^\infty.
\]
\end{theorem}

\begin{proof}
What is needed is to show that the induced maps
\[
\iota \otimes_{alg} \ell : E_{m+1}\otimes_{alg} F_{m+1} \rightarrow  E_{m}\otimes_{alg} F_{m}
\]
extend to injective, compact operators
\[
\iota \otimes_{\epsilon} \ell : E_{m+1}\otimes_{\epsilon} F_{m+1} \rightarrow  E_{m}\otimes_{\epsilon} F_{m}
\]
and that the set
\[
E_{\infty}\otimes_{\epsilon}F_{\infty} := \bigcap_{i=0}^\infty E_i \otimes_{\epsilon} F_i 
\]
is dense in $E_m \otimes_{\epsilon}F_m$ for all $m\geq 0$.

That $\iota \otimes_{\epsilon} \ell$ is compact was shown in \cite[Theorem 2]{holub1972compactness}, and the fact that $\iota \otimes_{\epsilon} \ell$ is injective is a well known consequence of the injective tensor product, see \cite[Chapter 43]{francois}.

Fix $\omega \in E_m\otimes_{alg} F_m$, then we may write the finite sum $\omega = \sum a_i\otimes b_i.$ For each $a_i \in E_m$ and $b_i \in F_m$ there exist sequences $\{a_{i_k}\}_k^\infty \subseteq E_\infty$ and $\{b_{i_k}\}_k^\infty \subseteq F_\infty$ such that $\lim_{k} a_{i_k} = a_i$ and $\lim_{k} b_{i_k} = b_i$, in each respective norm. Now consider the sequence $\{ \Tilde{\omega}_k\}\subseteq E_{\infty}\otimes_{alg} F_{\infty} $ defined by
\[
\Tilde{\omega}_k = \sum_i a_{i_k} \otimes b_{i_k}.
\]
Left to show is that $\Tilde{\omega}_k$ converges to $\omega$ in the injective tensor norm on $E_m\otimes_{alg}F_m$. Suppose, without loss of generality that 
\[
0<\sum_i \norm{a_{i_k}}_{E_m} \leq \sum_i \norm{a_i}_{E_m} = A <\infty, \quad \text{ and } \quad 0<\sum_i \norm {b_i}_{F_m} =B < \infty.
\]
Moreover, for $\varepsilon>0$ arbitrary, let $K\in \mathbb{N}$ be such that 
\[
\norm{a_i - a_{i_k}}_{E_m} <\frac{\varepsilon}{2B}, \quad \text{ and } \quad \norm{b_i -b_{i_k}}_{F_m} <\frac{\varepsilon}{2A},
\]
for all $k \geq K$ and all $i$. 
Then, using that the injective tensor product is a 
cross norm, i.e. $\norm{x\otimes y}_\varepsilon = \norm{x}_{E_m}\norm{y}_{F_m}$ for all $x\in E_m$ and all $y\in F_m$, it holds that
\[
\begin{aligned}
\norm{\omega - \Tilde{\omega}_k}_\epsilon &= \norm{\sum_i a_i\otimes b_i - a_{i_k} \otimes b_{i_k}}_\epsilon\\
&\leq \norm{\sum_i a_i\otimes b_i - a_{i_k} \otimes b_{i}}_\epsilon + \norm{\sum_i a_{i_k}\otimes b_i - a_{i_k} \otimes b_{i_k}}_\epsilon\\
&= \norm{\sum_i (a_i - a_{i_k} )\otimes b_{i}}_\epsilon + \norm{\sum_i a_{i_k}\otimes (b_i - b_{i_k})}_\epsilon\\
&\leq \sum_i \norm{a_i - a_{i_k}}_{E_m} \norm{b_i}_{F_m} + \sum_i \norm{a_{i_k}}_{E_m} \norm{b_i-b_{i_k}}_{F_m}\\
&< \frac{\varepsilon}{2B} \sum_i \norm{b_i}_{F_m} + \frac{\varepsilon}{2A} \sum_i \norm{a_{i_k}}_{E_m} \leq \varepsilon
\end{aligned}
\]
for all $k \geq K$.  This concludes the proof.
\end{proof}

\begin{notation}
Henceforth, we shall simplify the notation by denoting the sc-Banach space $E \otimes_\epsilon F$ with sc-structure as in Theorem \ref{thm:tensorscbanach} by $E\otimes F$.
\end{notation}

If $X$ is an M-polyfold we let 
$$
\bigotimes_1^rTX=\bigsqcup_{p\in X_1}\underbrace{T_p X \otimes\cdots\otimes T_pX}_{r \text{ times}}
$$
which is an M-polyfold in a natural way in view of Theorem~\ref{thm:tensorscbanach}.

We will use the convention that the empty tensor product is $\mathbb{R}$, i.e.,
\begin{equation}\label{convention2}
\bigotimes_1^0TX=\mathbb{R}.
\end{equation}

When considering tensors on an M-polyfold, requiring a certain kind of reflexivity is natural. A \emph{reflexive sc-Banach space} 
is an  sc-Banach space $E$ such that $E_n$ is reflexive for all $n\geq0$. An M-polyfold modeled on reflexive sc-Banach spaces is called a \emph{reflexive M-polyfold}.

\begin{definition}
\label{def:tensorfield}
Let $X$ be a reflexive M-polyfold, and define $\mathcal{T}^r_s(X)$ to be the set of sc-smooth maps
\[
\tau:  \bigoplus_{m=1}^s TX\rightarrow \bigotimes_{n=1}^r TX 
\]
such that for each $p\in X_1$, the map
\[
\tau: \bigoplus_{m=1}^s T_p X   \rightarrow \bigotimes_{n=1}^r T_pX
\]
is an $s$-multilinear sc-continuous map.
Elements of $\mathcal{T}^r_s(X)$ are called \emph{$(r,s)$-tensor fields}.
\end{definition}

Notice that with our conventions \eqref{convention1} and \eqref{convention2}, $(1,0)$-tensor fields are vector fields and $(0,1)$-tensor fields are differential $1$-forms.

Next, we shall introduce metrics on reflexive M-polyfolds. The model spaces are Banach spaces. Therefore, there are at least two plausible definitions of semi-Riemannian metrics on M-polyfolds; weak semi-Riemannian and strong semi-Riem\-ann\-ian metrics.

\begin{definition}
\label{def:riemmetric}
Let $X$ be a reflexive M-polyfold. We define the following properties for a symmetric tensor field $g \in \mathcal{T}^0_2(X)$:

\begin{enumerate}[$(i)$]
    \item For each $p \in X_1$, if
    \[
    g_p(v_p, w_p) = 0 \quad \text{for all } w_p \in T_p X,
    \]
    then $v_p = 0$.
    \item[$(i')$] For each $p \in X_1$, the map
    \[
    v_p \longmapsto g_p(v_p, \cdot)
    \]
    is a topological vector space isomorphism from $T_p X$ to $(T_p X)^*$, where $(T_p X)^*$ is the space of continuous functionals of the zeroth level of $T_pX$.
    \item[$(ii)$] For all $p \in X_1$ and all $v_p \in T_p X$ with $v_p \neq 0$,
    \[
    g_p(v_p, v_p) > 0.
    \]
\end{enumerate}

Using these properties, we define the following types of metrics on $X$:

\begin{enumerate}[$(a)$]\itemsep2mm
    \item A \emph{weak semi-Riemannian metric} is a symmetric tensor field $g \in \mathcal{T}^0_2(X)$ that satisfies property $(i)$.
    \item A \emph{strong semi-Riemannian metric} is a symmetric tensor field $g \in \mathcal{T}^0_2(X)$ that satisfies property $(i')$.
    \item A \emph{weak Riemannian metric} is a weak semi-Riemannian metric that additionally satisfies property $(ii)$.
    \item A \emph{strong Riemannian metric} is a strong semi-Riemannian metric that additionally satisfies property $(ii)$.
\end{enumerate}
\end{definition}

An M-polyfold modeled on Hilbert spaces, meaning that the zeroth scale of the underlying sc-Banach spaces is a Hilbert space, is called a \emph{Hilbert M-polyfold}. Note that an M-polyfold of class $\Gamma$ is reflexive and Hilbert as all finite scales are Hilbert spaces. We have the following result on the existence of strong Riemannian metrics for Hilbert M-polyfolds.

\begin{theorem}
\label{thm:existriemannian}
There exists a strong Riemannian metric on a reflexive Hilbert M-polyfold.
\end{theorem}

\begin{proof}
Choose charts $\{(U_\alpha,\phi_\alpha, (O_\alpha,C_\alpha,E_\alpha))\}_{\alpha \in \mathcal{A}}$ such that $\bigcup_\alpha U_\alpha=X$.
 For each $q\in U_\alpha$,  let $g_\alpha:TU_\alpha \oplus TU_\alpha \rightarrow \mathbb{R}$ be defined by 
\[
g_\alpha(q)(v_q,w_q) = 
\langle D \phi_\alpha(v_q),D\phi_\alpha(w_q) \rangle_\alpha
\]
where $v_q,w_q\in T_qU_{\alpha} \cong T_qX$ and $\langle\cdot,\cdot\rangle_\alpha$ denotes the inner product of $E_\alpha$. The mappings $g_\alpha$ are sc-smooth as they are infinitely Frech\'et differentiable and sc-continuous since $\mathbb{R}$ has the constant sc-structure. Furthermore, there exists an sc-smooth partition of unity $\{\psi_\alpha\}_{\alpha \in  \mathcal{A}}$ subordinate to the covering $\{U_\alpha\}_{\alpha \in \mathcal{A}}$. Then, the mapping $g:TX\oplus TX \rightarrow \mathbb{R}$ defined by
\[
g_q(v_q,w_q) = \sum_{\alpha\in \mathcal{A}} \psi_\alpha(q) g_\alpha(q)(v_q,w_q),
\]
where $q\in X_1$ and $v_q,w_q \in T_qX$, is a strong Riemannian metric. 
\end{proof}

\begin{remark}
Similar to Banach manifolds, the requirement that the modeling space is an sc-Hilbert space is 
necessary for strong Riemannian metrics.
\end{remark}

One can transfer almost complex structures from the traditional context to the M-polyfold framework, similar to adapting Riemannian metrics for M-polyfolds.

\begin{definition}
Let $X$ be an even-dimensional reflexive M-polyfold. If there is an sc-smooth $(1,1)$-tensor field $J:TX \rightarrow TX$ such that 
$J_p^2= -\id_{T_p X}$ for all $p\in X_1$, then the pair $(X,J)$ is called an \emph{almost complex M-polyfold}.
\end{definition}

Similar to \cite[Definition 2.1]{lempert1998dolbeault}, the almost complex structure $(X,J)$ induces a complex scalar multiplication $(a+bi,v)\mapsto av+bJ_pv$ where $a,b\in \mathbb{R}$ and $v\in T_pX$. 

 The mapping $J_p$,
extended $\mathbb{C}$-linearly to $\mathbb{C}\otimes T_pX$,
has eigenvalues $\pm i$ for all $p\in X_1$. Define $T^{1,0}_pX$ and $T^{0,1}_pX$ as the eigenspaces corresponding to $i$ and $-i$, respectively. Then $\mathbb{C}\otimes T_pX=T_p^{1,0}X\oplus T_p^{0,1}X$ for all $p\in X_1$ and
\begin{align}
\varpi_{1,0}& \colon \mathbb{C}\otimes T_pX\to T_p^{1,0}X,\quad
v\mapsto (v-iJv)/2,
\label{eq:varpihol}\\     
\varpi_{0,1}&\colon \mathbb{C}\otimes T_pX\to T_p^{0,1}X,\quad
v\mapsto (v+iJv)/2    
\label{eq:varpinatihol} 
\end{align}
are the natural projections on $T_p^{1,0}X$ and $T_p^{0,1}X$,
respectively.

Let
$$
T_\mathbb{C}^kX=\bigsqcup_{p\in X_1}\bigoplus_1^k\mathbb{C}\otimes T_pX, 
$$
and let $\Omega_{\mathbb{C}}^k(X)$ be the set of sc-smooth mappings 
$T_\mathbb{C}^{k}X\rightarrow \mathbb{C}$ that are $\mathbb{C}$-multilinear and alternating on 
$\mathbb{C}\otimes T_pX\times\cdots\times \mathbb{C}\otimes T_pX$ for $p\in X_1$.

\begin{definition}
If $\omega\in\Omega_{\mathbb{C}}^k(X)$, then we say that $\omega$ has bidegree $(s,k-s)$,
$\omega\in\Omega^{s,k-s}(X)$, if it has the following property for each $p\in X_1$:
If $v_1,\ldots,v_\ell\in T^{1,0}_pX$, $v_{\ell+1},\ldots,v_k\in T_p^{0,1}X$, and $\ell\neq s$,
then $\omega(v_1,\ldots,v_k)=0$. 
\end{definition}

It follows that we have a direct sum decomposition $\Omega_{\mathbb{C}}^k(X)=\oplus_{s+t=k}\Omega^{s,t}(X)$
and we let
$$
\pi_{s,t}\colon \Omega_{\mathbb{C}}^k(X)\to \Omega^{s,t}(X)
$$ 
be the natural projections.

\begin{example}
If $\omega\in\Omega_{\mathbb{C}}^1(X)$, then
$$
\pi_{1,0}\omega(v)=\omega(\varpi_{1,0} v),\quad
\pi_{0,1}\omega(v)=\omega (\varpi_{0,1} v).
$$
If $\omega\in\Omega^2_{\mathbb{C}}(X)$, then
\begin{eqnarray*}
\pi_{2,0}\omega(v_1,v_2) &=& \omega(\varpi_{1,0} v_1,\varpi_{1,0} v_2),\\
\pi_{1,1}\omega(v_1,v_2) &=& \omega(\varpi_{1,0} v_1,\varpi_{0,1} v_2)+\omega(\varpi_{0,1} v_1,\varpi_{1,0} v_2),\\
\pi_{0,2}\omega(v_1,v_2) &=& \omega(\varpi_{0,1} v_1,\varpi_{0,1} v_2).
\end{eqnarray*}
\qed
\end{example}

 If $\omega$ is a differential $k$-form as in Definition~\ref{def:scdifform}, then we extend it
to $T_\mathbb{C}^kX$ by extending it
to a $\mathbb{C}$-multilinear alternating mapping 
$\mathbb{C}\otimes T_pX\times\cdots\times \mathbb{C}\otimes T_pX\to \mathbb{C}$ for each $p\in X_1$. With this convention, $\Omega^k(X)\subset\Omega_{\mathbb{C}}^k(X)$.

\medskip

Now, we prove the existence of an almost complex structure on an even-dimensional M-polyfold of class $\Gamma$.

\begin{theorem}
\label{thm:almostcomplexmpoly}
Let $X$ be an even-dimensional M-polyfold of class $\Gamma$. Then, there exists an almost complex structure on $X$.
\end{theorem}

\begin{proof}
Assume that $X$ is an even-dimensional M-polyfold of class $\Gamma$, i.e., $X=\Gamma^k_n$ with $n$ odd and $k$ even,
where $\Gamma_n^k$ is given by \eqref{gammank}. We adopt the notation in the last part of Section~\ref{ssecMpoly} used in the construction of $\Gamma_n^k$. 

Take the functions $\gamma_1,\ldots,\gamma_k\in  L^2(\mathbb{R})$ to have mutually
disjoint supports, each of which has $L^2$-norm $1$. Observe that $\gamma_j$ and $\dot{\gamma}_j$ are orthogonal in $L^2(\mathbb{R})$ since $\gamma_j\dot{\gamma}_j$
is the derivative of $\gamma_j^2/2$; recall that $\delta_0=0$, cf.\ \eqref{eq:constructionscbanachspace}. 
Since $\delta_0=0$ and $\gamma_j$ have disjoint supports, it follows that 
$\dot{\gamma}_1(\cdot+e^{1/t}),\dots,\dot{\gamma}_k(\cdot+e^{1/t}), \gamma_1(\cdot+e^{1/t}),\dots,\gamma_k(\cdot+e^{1/t})$ all are mutually orthogonal.
Let 
\begin{equation}\label{glas}
\rho_{j,t}=\frac{\dd}{\dd t}\gamma_j(\cdot+e^{1/t}).
\end{equation}
Then $\gamma_1(\cdot+e^{1/t}),\dots,\gamma_k(\cdot+e^{1/t}), \rho_{1,t},\dots,\rho_{k,t}$ all are mutually orthogonal in 
$L^2(\mathbb{R})$.

The tangent map $Tr\colon T(\mathbb{R}^n\oplus\mathbb{R}\oplus L^2(\mathbb{R}))\to T(\mathbb{R}^n\oplus\mathbb{R}\oplus L^2(\mathbb{R}))$ is given by 
\[
Tr(x_1,\dots,x_n,t,f) (\delta x_1,\dots,\delta x_n,\delta t,\delta f) = (\delta x_1,\dots, \delta x_n,\delta t,0)
\]
whenever $t\leq 0$, and a straightforward calculation shows
that
\begin{multline*}
Tr(x_1,\dots,x_n,t,f) (\delta x_1,\dots,\delta x_n,\delta t,\delta f)\\
= \Big(\delta x_1,\dots,\delta x_n,\delta t, \sum_{j=1}^k\left\langle \delta f,\gamma_j(\cdot + e^{\frac{1}{t}})\right\rangle_{L^2}\gamma_j(\cdot + e^{\frac{1}{t}}) + \left\langle f,\rho_{j,t}\right\rangle_{ L^2}\gamma_j(\cdot + e^{\frac{1}{t}})\delta t \\
+  \left\langle f,\gamma_j(\cdot + e^{\frac{1}{t}})\right\rangle_{ L^2}\rho_{j,t}\delta t \Big).
\end{multline*}
Fix $p=r(x_1,\ldots,x_n,t,f)$ in $X$.
In view of \eqref{lunch}, we may assume that 
$f\in \spann\{\gamma_1(\cdot + e^{\frac{1}{t}}),\dots,\gamma_k(\cdot + e^{\frac{1}{t}})\}$.
Then $\langle f,\rho_{j,t}\rangle=0$ for all $j$.
Since $T_pX$ is the image of $Tr(x_1,\ldots,x_n,t,f)$
we obtain that $T_pX$ is the set of all
\begin{multline*}
 \Big(\delta x_1,\dots,\delta x_n,\delta t, \sum_{j=1}^k\left\langle \delta f,\gamma_j(\cdot + e^{\frac{1}{t}})\right\rangle_{L^2}\gamma_j(\cdot + e^{\frac{1}{t}}) + \left\langle f,\gamma_j(\cdot + e^{\frac{1}{t}})\right\rangle_{L^2}\rho_{j,t}\delta t \Big),
\end{multline*}
when $\delta x_j,\delta t\in\mathbb{R}$ and 
$\delta f\in L^2(\mathbb{R})$. We see that in fact it suffices to take 
$\delta f\in \spann\{\gamma_1(\cdot + e^{\frac{1}{t}}),\dots,\gamma_k(\cdot + e^{\frac{1}{t}})\}$.
Hence, if
\begin{equation}\label{sladdaren}
\xi_t=\sum_{j=1}^k\left\langle f,\gamma_j(\cdot + e^{\frac{1}{t}})\right\rangle_{L^2}\rho_{j,t},
\end{equation}
then $T_pX$ is the set of all 
\begin{equation}\label{sladden}
(\delta x_1,\dots,\delta x_n,\delta t,\delta f + \xi_t 
\delta t)
\end{equation}
when $\delta x_j,\delta t\in\mathbb{R}$ and 
$\delta f\in \spann\{\gamma_1(\cdot + e^{\frac{1}{t}}),\dots,\gamma_k(\cdot + e^{\frac{1}{t}})\}$.
Hence, the representation $\delta f + \xi_t \delta t$ is unique, since $\delta f$ and $\xi_t$ are orthogonal. We then define $J_p\colon T_pX\to T_pX$ by
\begin{equation}\label{kota}
J_p(\delta x_1,\dots,\delta x_n,\delta t, 0) = (-\delta x_2, \delta x_1, \dots,-\delta x_{n-1},\delta x_{n-2},  -\delta t,\delta x_n,0)
\end{equation}
whenever $t\leq 0$, and
\begin{multline}\label{sladdare} 
J_p(\delta x_1,\dots,\delta x_n,\delta t,\delta f+\xi_t\delta t)\\
= \Big(-\delta x_2, \delta x_1, \dots,-\delta x_{n-1},\delta x_{n-2},  -\delta t,\delta x_n, \sum_{j=1}^{k/2}\left\langle \delta f,\gamma_{2j-1}(\cdot + e^{\frac{1}{t}}) \right\rangle_{L^2} \gamma_{2j}(\cdot + e^{\frac{1}{t}})\\
-  \left\langle \delta f,\gamma_{2j}(\cdot + e^{\frac{1}{t}}) \right\rangle_{L^2} \gamma_{2j-1}(\cdot + e^{\frac{1}{t}}) + \xi_t\delta x_n \Big)
\end{multline}
whenever $t>0$. It is straightforward to check that  $J_p^2= -\id_{T_p X}$ for all $p\in X_1$.

Let $J\colon TX\to TX$ be the mapping which fiber-wise is
$J_p$. It follows that $J$ is sc-smooth by arguments similar to those showing that $r$ is sc-smooth. Thus, $J$ is an almost
complex structure on $X$.
\end{proof}

Recall from Remark~\ref{rem: Gamma}  that two $\Gamma^k_n$ defined by different choices of $\gamma_j$ are naturally sc-diffeomorphic. Endowing both these M-polyfolds with the corresponding almost complex structure given in the proof of Theorem~\ref{thm:almostcomplexmpoly} gives two equivalent almost complex M-polyfolds.

Let us discuss how to use Theorem~\ref{thm:almostcomplexmpoly} to generate a significantly larger class of almost complex M-polyfolds. The first approach involves the sc-Banach space $E$ defined by~\eqref{eq:constructionscbanachspace}. We then derive a non-trivial infinite-dimensional, almost complex M-polyfold by considering the direct sum:
\begin{equation}
\label{eq:infiniteGamma}
(\mathbb{C}\otimes E)\oplus \Gamma^{k}_n
\end{equation}
together with the induced sc-structure. 

Secondly, M-polyfolds of class $\Gamma$ may be glued together to form more general M-polyfolds. Such M-polyfolds may be finite-dimensional or infinite-dimensional M-polyfolds by considering charts on spaces as in~\eqref{eq:infiniteGamma}. M-polyfolds that are sc-smoothly glued together using M-polyfolds of class $\Gamma$ as local models need not admit an almost complex structure. Thus, no general existence result as in Theorem~\ref{thm:almostcomplexmpoly} holds globally.

In~\cite[Section 3.4]{hofer2010sc}, the Cauchy-Riemann operators, $\partial$ and $\bar{\partial}$, were utilized on a specific M-polyfold bundle. In this context, we aim to define these operators on an almost complex M-polyfold using the classical approach.

\begin{definition}\label{def:dbar}
Let $X$ be an almost complex M-polyfold. First consider the complexified exterior derivative
\[
d_\mathbb{C}: \Omega^k_{\mathbb{C}}(X)\rightarrow \Omega^{k+1}_{\mathbb{C}}(X^1)
\]
restricted to the subspace $\Omega^{p,q}(X)\subset \Omega^k_{\mathbb{C}}(X)$. As before, let 
\[
\pi_{p+1,q}:\Omega^{k+1}_{\mathbb{C}} (X^1)\rightarrow \Omega^{p+1,q}(X^1),
\]
and
\[
\pi_{p,q+1}:\Omega^{k+1}_{\mathbb{C}} (X^1)\rightarrow \Omega^{p,q+1}(X^1)
\]
be the natural projections. Define the maps
\[
\partial : \Omega^{p,q} (X^1) \rightarrow \Omega^{p+1,q}(X^1),\qquad  \partial =\pi_{p+1,q} \circ d_\mathbb{C},
\]
and
    \[
\bar{\partial}: \Omega^{p,q}(X) \rightarrow \Omega^{p,q+1}(X^1), \qquad  \bar{\partial} =\pi_{p,q+1} \circ d_\mathbb{C}.
\]
\end{definition}

\begin{example}\label{d-exempel}
Let $(X,J)$ be an almost complex M-polyfold. If $f$ is an sc-smooth function (real- or complex-valued) on $X$, then
\begin{align*}
(\bar\partial f)_p (v)& =\frac{\big((\dd_{\mathbb{C}}f)_p (v) + i (\dd_{\mathbb{C}}f)_p(Jv)\big)}{2},\; \text{ and }\\
(\partial f)_p (v)&=\frac{\big((\dd_{\mathbb{C}}f)_p (v) - i( \dd_{\mathbb{C}}f)_p(Jv)\big)}{2}
\end{align*}
for all $p\in X_1$ and all $v\in T_pX$.
\qed
\end{example}

Let $(X,J)$ be an almost complex M-polyfold. A weak (resp.\  strong) Riemannian metric $g$ defined on $X$ is said to be \emph{compatible} with the almost complex structure $J$ if the following equation holds for all tangent vectors
$v_p, w_p \in T_pX$, and all $p\in X_1$: 
\[
g_p(v_p,w_p) =g_p(Jv_p,Jw_p).
\]
For any arbitrary Riemannian metric $g$ on $(X, J)$, we can construct a compatible Riemannian metric, denoted as $\widetilde{g}$, by letting
\[
\widetilde{g}_p(v_p,w_p) = g_p(v_p,w_p) + g_p(Jv_p,Jw_p), 
\]
for all $v_p, w_p \in T_p X$, and all $p\in X_1$. 

The induced standard inner product from the  direct sum $\mathbb{R}^n\oplus \mathbb{R}\oplus L ^2(\mathbb{R})$ is not compatible with the almost complex structure constructed in Theorem~\ref{thm:almostcomplexmpoly}. 

In the literature, a compatible Riemannian metric is sometimes referred to as an \emph{almost Hermitian metric}. However, we reserve this terminology for the corresponding complex-valued sesquilinear metric.

\begin{definition}\label{DEF: Almost Hermitian Metric}
Let $(X, J)$ be an almost complex M-polyfold. Let $h: TX \oplus TX \to \mathbb{C}$ be a sc-smooth $(0,2)$-tensor field such that for all $v_p, w_p \in T_p X$, $p \in X_1$, and all $a, b \in \mathbb{C}$:
\begin{enumerate}[$(i)$]\itemsep2mm
    \item $h_p(av_p, w_p) = a h_p(v_p, w_p)$ and $h_p(v_p, bw_p) = \bar{b} h_p(v_p, w_p)$;
    \item $h_p(v_p, w_p) = \overline{h_p(w_p, v_p)}$;
    \item $h_p(v_p, w_p) = h_p(Jv_p, Jw_p)$.
\end{enumerate}
If, in addition, 
\begin{enumerate}[$(iv)$]
    \item $h_p(v_p, v_p) > 0$ for all non-zero $v_p \in T_p X$,
\end{enumerate} 
then $h$ is called a \emph{weak almost Hermitian metric}. Furthermore, if additionally 

\begin{enumerate}[$(iv')$]
    \item the map $v_p \longmapsto h_p(v_p, \cdot)$ is a topological vector space isomorphism from $T_p X$ to $\mathbb{C}\otimes (T_p X)^*$
    for all $p \in X_1$,
\end{enumerate}
then $h$ is called a \emph{strong almost Hermitian metric}. An almost complex M-polyfold $(X,J)$ equipped with a weak (resp.\ strong) almost Hermitian metric $h$ is called a \emph{weak (resp.\ strong) almost Hermitian M-polyfold}. 
\end{definition}

As in classical almost complex geometry there is the corresponding fundamental 2-form associated to an almost complex structure with compatible Riemannian metric (see e.g. \cite{falcitelli1994almost}).

\begin{definition}
\label{def:fundamental2form}
Consider an almost complex M-polyfold $(X, J)$ equipped with a compatible weak (resp.\ strong) Riemannian metric $g$. The \emph{fundamental form induced by $(J,g)$} is the sc-smooth 2-form $\omega \in \Omega^2(X)$ defined by
\[
\omega_p(v_p, w_p) := g_p(Jv_p, w_p)
\]
for all  $v_p, w_p \in T_p X$ at each  $p \in X_1$. If $\omega$ is closed, it is called the \emph{weak (resp.\ strong) symplectic form induced by $(J,g)$}.
\end{definition}

Symplectic forms may be defined independently of any Riemannian metric and almost complex structure.  We say that a \emph{weak symplectic form} on an M-polyfold $X$ is a closed, sc-smooth 2-form $\omega \in \Omega^2(X)$ such that the map $v_p \mapsto \omega_p(v_p, \cdot)$ is injective for each $v_p \in T_p X$, $p \in X_1$. By contrast, a \emph{strong symplectic form} requires this map to be a topological vector space isomorphism from $T_p X$ to $(T_p X)^*$ at each $p \in X_1$. However, this paper derives the symplectic form $\omega$ from a compatible Riemannian metric $g$, as it will be needed in Theorem~\ref{thm:nonclosedsymplectic}.

It is worth noting that an almost complex M-polyfold $(X,J)$ equipped with a compatible metric $g$ allows the construction of an almost Hermitian metric. Specifically, one can define the $\mathbb{C}$-valued $(0,2)$-tensor field $h$ by the expression:
\[
h_p(v_p,w_p) = g_p(v_p,w_p) + i\omega_p(v_p,w_p),
\]
where $\omega$ denotes the fundamental form induced by $(J,g)$.

Moving on, we shall demonstrate the existence of symplectic forms on M-polyfolds of class $\Gamma$ with even dimensions.

\begin{theorem}
\label{thm:nonclosedsymplectic}
Let $X$ be an even-dimensional M-polyfold of class $\Gamma$. There exists an almost complex structure $J$ and a compatible strong Riemannian metric on $(X,J)$ such that the induced fundamental form is a symplectic form.
\end{theorem}

\begin{proof}
We continue with the notation from the proof of Theorem~\ref{thm:almostcomplexmpoly}. In particular,
$X=\Gamma_n^k$ and $J$ is the almost complex structure defined by \eqref{sladdare}.
Consider a point $p=r(x_1,\ldots,x_n,t,f)$ in $X$.  Recall that if $t>0$, then the tangent space $T_pX$ is the set of all 
$$
\big(\delta x_1,\ldots,\delta x_n,\delta t, \delta f + \xi(t,f)\delta t\big)
$$
when $\delta x_j,\delta t\in\mathbb{R}$ and 
$\delta f\in \spann\{\gamma_1(\cdot + e^{\frac{1}{t}}),\dots,\gamma_k(\cdot + e^{\frac{1}{t}})\}$, and where
$\xi(t,f)=\xi_t$ is given by \eqref{sladdaren}. Recall also that $\xi(t,f)$ is $L^2$-orthogonal to each 
$\gamma_j(\cdot + e^{\frac{1}{t}})$. If $t\leq 0$, then $T_pX$ is the set of all
$(\delta x_1,\ldots,\delta x_n,\delta t, 0)$.

Define the sc-smooth vector fields  $\mu_1,\ldots,\mu_{n+k+1}$ on $X$ as follows. For $1\leq \ell\leq n$, 
\begin{equation}\label{mu}
\mu_\ell=(0,\ldots 0,1,0,\ldots,0),
\end{equation}
where $1$ is in the $\ell$th slot,
\begin{equation}\label{muu}
\mu_{n+1}=
\begin{cases}
(0,\ldots,0,1,0), & t\leq 0 \\
\big(0,\ldots,0,1,\xi(t,f)\big), & t>0
\end{cases},
\end{equation}
and
\begin{equation}\label{muuu}
\mu_{n+1+j}=
\begin{cases}
\big(0,\dots,0,0\big), & t\leq 0 \\
\big(0,\ldots,0,\gamma_j(\cdot+e^{1/t})\big), & t>0
\end{cases}
\end{equation} 
for $j=1,\ldots,k$. It follows that $\{\mu_\ell(p)\}_\ell$ is a bases for $T_pX$, and $\{\mu_\ell\}_\ell$ is a frame field for $TX$. We define a Riemannian metric $g_p$ on $T_pX$ declaring that $\{\mu_\ell(p)\}$ is orthonormal. Explicitly this Riemannian metric is given by 
\begin{multline*}
g_p\big((\delta x_1,\dots,\delta x_n,\delta t,\delta f),(\delta y_1,\dots,\delta y_n,\delta s,\delta \tilde{f})\big) \\
=\sum_{i=1}^n \delta x_i \delta y_i + \delta t\delta s + \sum_{j=1}^k \innerprod{\delta f, \gamma_j(\cdot + e^{\frac{1}{t}})}_{L^2}\innerprod{\delta \tilde{f}, \gamma_j(\cdot + e^{\frac{1}{t}})}_{L^2}.
\end{multline*}

Next, we check that $g$ is compatible with $J$. A direct calculation, using \eqref{sladdare}, shows that 
\begin{equation}\label{bula}
J\mu_\ell=
\begin{cases}
\mu_{\ell+1}, & \ell\,\,\text{odd} \\
-\mu_{\ell-1}, & \ell\,\,\text{even}
\end{cases}.
\end{equation}
Hence, $g(J\mu_j,J\mu_\ell)=g(\pm\mu_{j\pm 1},\pm\mu_{\ell\pm 1})=\delta_{j \ell}=g(\mu_j,\mu_\ell)$. 
Since $\{\mu_\ell\}$ is a frame field, we conclude that $g$ is compatible with $J$.

To complete the proof, we now show that the fundamental $2$-form $\omega$ induced by $(J,g)$ is $d$-closed. 
Because $\{\mu_\ell\}$ is a frame field, it suffices to check that $d\omega (\mu_i,\mu_j,\mu_\ell)=0$.
From~\eqref{bula} and $\omega(\mu_j,\mu_\ell)=g(J\mu_j,\mu_\ell)$, together with  $g(\mu_{j},\mu_\ell)=\delta_{j\ell}$ 
we see that $\omega(\mu_j,\mu_\ell)$ is constant (taking values in $\{0,1,-1\}$). Hence, $D \omega(\mu_j,\mu_\ell)=0.$ 
We claim that 
\begin{equation}\label{claim}
[\mu_i,\mu_\ell]=0.
\end{equation}
Assuming this for the moment, we deduce 
$d\omega (\mu_i,\mu_j,\mu_\ell)=0$ in view of Definition~\ref{def:extdifferential}.

It remains to prove the claim. From~\eqref{mu},~\eqref{muu}, and~\eqref{muuu} it is clear that $D\mu_\ell=0$
whenever $t\leq 0$. For $t>0$ a straightforward calculation shows that $D\mu_\ell=0$ when $1\leq\ell\leq n$,
\begin{equation}\label{hustak}
D\mu_{n+1} :
(\delta x_1,\ldots,\delta x_n,\delta t,\delta f+\xi(t,f)\delta t)\mapsto
\left(0,\ldots,0,\xi(t,\delta f)+\frac{\dd\xi}{\dd t}(t,f)\delta t\right),
\end{equation}
and
$$
D\mu_{n+1+j} : (\delta x_1,\ldots,\delta x_n,\delta t,\delta f+\xi(t,f)\delta t)\mapsto
(0,\ldots,0,\rho_{j,t}\delta t)
$$
when $1\leq j\leq k$, where $\rho_{j,t}$ is given by \eqref{glas}.
In \eqref{hustak}, we used the linearity of $\xi(t,f)$ in $f$ and the fact that  $\xi(t,\xi(t,f))=0$.
It follows that
$$
D\mu_{n+1} . \mu_{n+1+j}=\big(0,\ldots,0,\xi(t,\gamma_j(\cdot+e^{1/t}))\big)
=(0,\ldots,0,\rho_{j,t})
$$
and
$$
D\mu_{n+1+j} . \mu_{n+1} = (0,\ldots,0,\rho_{j,t})
$$
for $j=1,\ldots,k$. 
Hence, by \eqref{lie}, we conclude that
$$
[\mu_{n+1+j}, \mu_{n+1}] \;=\; 0,
$$
which proves the claim for these indices.
The remaining cases follow directly from the expressions for $D\mu_\ell$. This settles the claim, and thus completes the proof.
\end{proof}

\section{Complex M-polyfolds}\label{sec:complex}

Similarly to previous sections, we continue to assume that all sc-Banach spaces are over $\mathbb{R}$. 

Taking inspiration from~\cite[Definition 2.1]{lempert1998dolbeault}, we shall introduce the concept of a complex M-polyfold as an integrable, almost complex M-polyfold.   In Corollary~\ref{cor:complex}, we confirm the integrability of our almost complex structure presented in Theorem~\ref{thm:almostcomplexmpoly}, deriving our first family of complex M-polyfolds.  In Proposition~\ref{prop:complex}, we establish the existence of a rich structure in the space of sc-holomorphic functions on the complex M-polyfolds described in Corollary~\ref{cor:complex}. We conclude by constructing the complex M-polyfold shown in Figure~\ref{fig:complexC2Mpolyfold} on page~\pageref{fig:complexC2Mpolyfold}.

\medskip

A fundamental property associated with complex geometry is the integrability of an almost complex structure. However, determining how to define this property within the context of M-polyfolds is not immediately evident. We present a definition below, followed by a discussion on the topic.

\begin{definition}
\label{def:integrable}
 We say that $(X,J)$ is \emph{integrable} if 
$[\mu,\nu]$ is an sc-smooth local section of $T^{1,0}X^1$ for 
all sc-smooth local sections
$\mu$ and $\nu$ of $T^{1,0}X$.  If $(X,J)$ is integrable, then we say that $(X,J)$ is a 
\emph{complex M-polyfold}.
\end{definition}

In classical complex geometry, numerous equivalent notions exist to define integrability. One of these is the vanishing of the Nijenhuis tensor. The \emph{Nijenhuis tensor} $N_J: TX \oplus TX\rightarrow TX^1$, is defined as follows: Given $v,w\in T_p X$, $p\in X_2$, take any local vector fields $\mu,\nu$ such that $v=\mu_p$ and $w=\nu_p$. Then
\[
N_J(\mu_p,\nu_p):= [\mu,\nu]_p + J[\mu,J\nu]_p + J[J\mu,\nu]_p - [J\mu,J\nu]_p.
\]
Notice that the Nijenhuis tensor is only a  tensor field in the sense of Definition~\ref{def:tensorfield} when $X$ is finite-dimensional  since the codomain of $N_j$ is not equal to $TX$ when $X$ is infinite-dimensional. It follows by a direct generalization of the classical proof that integrability in the 
sense of Definition~\ref{def:integrable} is equivalent to $N_J =0$.

\begin{cor}
\label{cor:complex}
Let $X$ be an even-dimensional M-polyfold of class $\Gamma$ equipped with the almost complex structure $J$ of Theorem~\ref{thm:almostcomplexmpoly}. Then $(X,J)$ is a complex M-polyfold.
\end{cor}

\begin{proof} We continue with the same notation as in the proof of  Theorem~\ref{thm:almostcomplexmpoly}. In particular, let  $X=\Gamma_n^k$ and let $J$ be the almost complex structure on $X$ there. Recall that the vector fields $\mu_\ell$ from \eqref{mu}, \eqref{muu}, and \eqref{muuu} form an sc-smooth frame field
of $TX$. By \eqref{bula} and \eqref{claim}, one verifies that  $N_J(\mu_i,\mu_\ell)=0$ for all $i,\ell$. 
Since $\{\mu_\ell\}$ is a frame field, we conclude that $N_J=0$, and so $(X,J)$ is a complex M-polyfold.
\end{proof}

\begin{definition}
\label{def:kahlermpolyfold}
Let $(X, J)$ be a complex M-polyfold. A \emph{weak (resp.\ strong) K\"{a}hler  structure} on $(X, J)$ is a pair $(g, \omega)$ given by a compatible weak (resp.\ strong) Riemannian metric $g$ and its induced symplectic form $\omega$. The quadruple $(X, J, g, \omega)$ is referred to as a \emph{weak (resp.\ strong) K\"{a}hler M-polyfold}.
\end{definition}

Let $(X, J, g, \omega)$ be as given by Theorem~\ref{thm:nonclosedsymplectic}. It follows from Theorem~\ref{thm:nonclosedsymplectic} and Corollary~\ref{cor:complex} that this quadruple forms a strong K\"{a}hler M-polyfold.

\begin{definition}\label{def:holom}
Let $X$ and $Y$ be almost complex M-polyfolds. Let $f:X\rightarrow Y$ be an sc$^1$-mapping. The \textit{complexified tangent map} $T_{\mathbb{C}}f$ is the tangent map as in Definition \ref{def: differential between Mpolyfolds} but $\mathbb{C}$-linearly extended to the complexified tangent bundles. Then, $f$ is said to be \emph{sc-pseudo-holomorphic} if
\[
 \varpi_{0,1}  \circ T_{\mathbb{C}} f \circ \varpi_{1,0} =0,
\]
where $\varpi_{1,0}$ and $\varpi_{0,1}$ are given in \eqref{eq:varpihol} and \eqref{eq:varpinatihol} respectively. If $X$ and $Y$ are complex M-polyfolds we say that $f$ is \emph{sc-holomorphic}.
\end{definition}

Notice that if $(X,J_X)$ and $(Y,J_Y)$ are almost complex M-polyfolds. Definition~\ref{def:holom} of sc-pseudo-holomorphic function is equivalent to saying that for $f\colon X\rightarrow Y$ satisfies
$$
Tf\circ J_X=J_Y\circ Tf.
$$
Equivalently, the following diagram commutes:
\[
\begin{tikzcd}
    TX \arrow[r,"Tf"] \arrow[d,"J_X"] & TY \arrow[d,"J_Y"]\\ TX \arrow[r,"Tf"] &TY
\end{tikzcd}
\]

Let $(X,J)$ be an almost complex M-polyfold and suppose that $f\colon X\to\mathbb{C}$ is
sc$^1$-smooth. In view of Example~\ref{d-exempel} we have $(\partial f)_p (Jv)=i(\partial f)_p (v)$ for all $p\in X_1$ and all $v\in T_pX$. Since $d_\mathbb{C}f=\partial f+\bar\partial f$ it follows that $d_\mathbb{C}f\circ J = i\cdot d_\mathbb{C} f$ if and only if
$\bar\partial f=0$. 
Thus, $f$ is sc-pseudo-holomorphic if and only if
$\bar\partial f=0$.

\begin{example}\label{embed}
Let $\iota \colon \Gamma_1^2 \hookrightarrow \mathbb{C}^2$ be the mapping  
defined by $\iota(r(x,t,f)) = (x+it, 0)$ when $t \leq 0$ and  
$$
\iota(r(x,t,f)) = \Big(x+it,\, \langle f, \gamma_1({\cdot}+e^{1/t})\rangle + i\langle f, \gamma_2({\cdot}+e^{1/t})\rangle\Big)
$$  
when $t > 0$.  
Clearly, $\iota$ is injective, and as in~\cite[Lemma 1.23]{hofer2010sc} one shows that it is sc$^1$-smooth.  
The differential of $\iota$ at $r(x,t,f)$ is the mapping  
$(\delta x, \delta t, \delta f + \xi_t \delta t) \mapsto (\delta x + i\delta t, 0)$ when $t \leq 0$ and  
$$
(\delta x, \delta t, \delta f + \xi_t \delta t) \mapsto \Big(\delta x + i\delta t,\, \langle \delta f, \gamma_1({\cdot}+e^{1/t})\rangle + i\langle \delta f, \gamma_2({\cdot}+e^{1/t})\rangle\Big)
$$  
when $t > 0$; cf.\ \eqref{sladdaren} and \eqref{sladden} for the notation. Since 
\[
\delta f \in\spann\{\gamma_1({\cdot}+e^{1/t}), \gamma_2({\cdot}+e^{1/t})\},
\]
it follows that the differential of $\iota$ is injective.  
Moreover, in view of \eqref{kota} and \eqref{sladdare},  
it follows that $T\iota \circ J = i \cdot T\iota$, where $J$ is the  
complex structure on $\Gamma_1^2$. Hence,  
$\iota$ is an sc-holomorphic embedding of  
$\Gamma_1^2$ into $\mathbb{C}^2$. The image is the set  
\begin{equation}\label{pennan}
\iota(\Gamma_1^2) = \{(z,w): \operatorname{Im} z \leq 0,\, w=0\} \cup \{(z,w): \operatorname{Im} z > 0\}.
\end{equation}
\qed
\end{example}

We will now present some fundamental properties of sc-holomorphic maps.

\begin{prop}\label{prop:complex}
\begin{enumerate}[$(i)$]\itemsep2mm
\item If $X$ is an even-dimensional M-polyfold of class $\Gamma$ equipped with the integrable almost complex structure of Theorem~\ref{thm:almostcomplexmpoly}, then
the vector space of sc-holomorphic functions $\phi: X \rightarrow \mathbb{C}$ is infinite-dimensional.

\item  Let $\iota: \Gamma_1^2 \hookrightarrow \mathbb{C}^2$ be the sc-holomorphic embedding defined in Example~\ref{embed}. Then, there exists an sc-holomorphic function $\phi: \Gamma_1^2 \to \mathbb{C}$ such that for any neighborhood $U$ of $0$ in $\mathbb{C}^2$, there is no holomorphic function $U \to \mathbb{C}$ that coincides with $\phi \circ \iota^{-1}$ on $U \cap \iota(\Gamma_1^2)$.

\item An sc-holomorphic function on a complex M-polyfold is not necessarily sc$^2$.
\end{enumerate}
\end{prop}

\begin{proof} $(i)$: Once again we consider $(X,J)$ as in Theorem~\ref{thm:almostcomplexmpoly}. Consider the points $p=r(x_1,\dots,x_n,t,f)\in X$, cf.\ \eqref{lunch}. Let $\mathcal{O}(\mathbb{C})$ denote the set of all holomorphic functions on $\mathbb{C}$. Then for any $\psi\in \mathcal{O}(\mathbb{C})$ the map
\[
\phi: X\rightarrow \mathbb{C}
\]
defined by
\[
\phi(p)= \psi(x_n +it)
\]
is sc-holomorphic. In fact, since $\psi$ is holomorphic, it is sufficient to check that the map
\[
\pi: X \rightarrow \mathbb{C}, \qquad
\pi(p)=x_n +it,
\]
is sc-holomorphic. It is immediate that $\pi$ is sc$^1$ by \cite[Proposition 2.1]{hofer2010sc},
and it is straightforward to check that 
$$
(T\pi)_p(J w_p) = i (T\pi)_p (w_p)
$$
for all $w_p \in T_p X$, $p\in X_1$.
Thus $\pi$ is sc-holomorphic. Since $\mathcal{O}(\mathbb{C})$ is infinite-dimensional, the space of sc-holomorphic mappings
$X\to \mathbb{C}$ is infinite-dimensional.

\medskip

$(ii)$: Let $\log$ be the complex logarithm with branch cut along the negative imaginary axis, let
$
\psi(z,w)=z^2w^2\log(z),
$
and define 
$$
\phi\colon\Gamma_1^2\to\mathbb{C},\qquad
\phi=\psi\circ \iota.
$$
Then $\phi$ is sc$^1$ in view of 
\cite[Lemma 1.23]{hofer2010sc}. Since $\iota$ is sc-holomorphic and $\psi$ is holomorphic when 
$\im(z)>0$ it follows by the chain rule that
$T\phi\circ J=i\cdot T\phi$ when $t>0$. If $t\leq 0$, then 
$\phi=0$ and thus $T\phi\circ J=i\cdot T\phi$ also when
$t\leq 0$.

Hence, $\phi$ is sc-holomorphic. 
However, cf.~\eqref{pennan},
$$
\phi\circ \iota^{-1}(z,w)=
\begin{cases}
z^2w^2\log(z), & (z,w)\in \iota(\Gamma_1^2), \,\,\im(z)>0\\
0, & (z,w)\in \iota(\Gamma_1^2),\,\, \im(z)\leq 0
\end{cases},
$$
which clearly has no holomorphic extension to any neighborhood of $0$ in $\mathbb{C}^2$.

$(iii)$: The statement follows as a consequence of the proof of $(ii)$, by noticing that the map $\phi$ has discontinuous second derivative.
\end{proof}

If we in the proof of (2) take $\psi$ to be holomorphic in $U\subset\mathbb{C}^2$, then $\phi=\psi\circ \iota$ is sc-holomorphic in $\iota^{-1}(U)$.
Thus, if we identify 
$\Gamma_1^2$ with its image 
$\iota(\Gamma_1^2)\subset\mathbb{C}^2$, then the stalk at $0$ of
the sheaf of germs of sc-holomorphic functions on
$\Gamma_1^2$ contains the stalk $\mathcal{O}_0(\mathbb{C}^2)$. By part (2) of the proposition this containment is strict.

\smallskip

Thus far, the complex M-polyfolds under consideration have been local in the sense that they are sc-smooth retracts. To extend our understanding, it is natural to ask whether it is possible to glue two complex M-polyfolds of class~$\Gamma$ in a non-trivial way so as to obtain a global complex M-polyfold. The following example shows that this is indeed possible.

\begin{example}\label{ex:complexC2Mpolyfold}
Consider the following subsets of $\mathbb{C}^2$,
\[
C=\{(z,w)\in \mathbb{C}^2: |z|<1\},
\]
and
\[
P=\{(z,w)\in \mathbb{C}^2: w=0\}.
\]
We define $X=C\cup P$ with the complex structure inherited from $\mathbb{C}^2$. We shall now verify that $X$ is a complex M-polyfold.  An illustration of the slice $\im(w)=0$ of $X$ is provided in Figure~\ref{fig:complexC2Mpolyfold} on page~\pageref{fig:complexC2Mpolyfold}. 

To see this, we cover $X$ by $U=X\setminus \{(1,0)\}$ and $V=X\setminus \{(-1,0)\}$ and show that there are bijective mappings $\Phi\colon U\to \Gamma_1^2$ and $\Psi\colon V\to\Gamma_1^2$ such that $\Phi\circ \Psi^{-1}$ and $\Psi\circ \Phi^{-1}$ are sc-holomorphic. It then follows that $X$ is a complex M-polyfold as in Definition~\ref{def:integrable}; this is the trivial direction of the 
Newlander-Nirenberg theorem.

Let $\phi(z)=(1+z)/(1-z)=\tau(z)+i\sigma(z)$, where 
$\tau=\text{Re}\, \phi$ and $\sigma=\im\, \phi$. This is a M\"{o}bius transformation that maps
$\{|z|<1\}$ onto the right half-plane, $\{|z|=1\}\setminus \{1\}$ onto the imaginary axis, and 
$\{|z|>1\}$ onto the left half-plane. Let $\Phi\colon U\to\Gamma_1^2,$ be defined by
$$
\Phi(z,w)=
\big(-\sigma(z),\tau(z),\Real(w)\gamma_1(\cdot+e^{1/\tau(z)})+\im(w)\gamma_2(\cdot+e^{1/\tau(z)})\big),
$$
where $\Real(w)\gamma_1(\cdot+e^{1/\tau(z)})+\im(w)\gamma_2(\cdot+e^{1/\tau(z)})$ is understood to be $0$
when $w=0$. It is straightforward to check that $\Phi$ is a bijection. Moreover, using that
$\tau$ and $\sigma$ satisfy the Cauchy-Riemann equations one checks that
$$
D\Phi\circ i=J\circ D\Phi,
$$
where $J$ is given by \eqref{kota} and \eqref{sladdare}. Thus $J$ is compatible with the complex structure 
on $U$ induced by the one on $\mathbb{C}^2$. 

Similarly, let $\psi(z)=(1-z)/(1+z)=\phi(-z)$, which is a M\"{o}bius transformation mapping $\{|z|<1\}$ onto
the right half-plane, $\{|z|=1\}\setminus \{-1\}$ onto the imaginary axis, and $\{|z|>1\}$ onto the left
half-plane. Define $\Psi\colon V\to\Gamma_1^2$ by $\Psi(z,w)=\Phi(-z,w)$. Then $\Psi$ is a bijection
making the complex structures on $V$ and $\Gamma_1^2$ compatible.

If $f\in\text{span}\{\gamma_1(\cdot+e^{1/t}),\gamma_2(\cdot+e^{1/t})\}$,  a calculation shows that
$$
\Phi\circ\Psi^{-1}(x,t,f)
=(-\sigma(x,t),\tau(x,t),F(x,t,f)),
$$
where
\begin{align}\label{kylling}
\sigma(x,t)=\sigma\circ\psi^{-1}(t-ix)=x/(t^2+x^2),\\
\tau(x,t)=\tau\circ\psi^{-1}(t-ix)=t/(t^2+x^2),\label{broiler}
\end{align}
and
$$
F(x,t,f)=\big\langle f,\gamma_1(\cdot+e^{1/t})\big\rangle\gamma_1(\cdot+e^{1/\tau(x,t)})
+\big\langle f,\gamma_2(\cdot+e^{1/t})\big\rangle\gamma_2(\cdot+e^{1/\tau(x,t)}).
$$
One shows that $\Phi\circ\Psi^{-1}$ is sc$^1$ in a similar way as in~\cite[Lemma 1.23]{hofer2010sc}. 
In the same way, $\Psi\circ\Phi^{-1}$ is seen to be sc$^1$.

It remains to see that $D(\Phi\circ\Psi^{-1})$ and $D(\Psi\circ\Phi^{-1})$ commute with the complex structures.
Recall that the tangent space of $\Gamma_1^2$ at $(x,t,f)$ is the set of all vectors of the 
form \eqref{sladden} (with $n=1$), where $\xi_t=\xi(t,f)$ is given by \eqref{sladdaren}.
If $\delta f\in\text{span}\{\gamma_1(\cdot+e^{1/t}),\gamma_2(\cdot+e^{1/t})\}$,
a calculation gives
\begin{equation*}
D(\Phi\circ\Psi^{-1})_{(x,t,f)}(\delta x,\delta t,\delta f + \xi(t,f)\delta t)=
 \big(-\delta \sigma, \delta\tau, F(x,t,\delta f)+\xi(\tau,F)\delta\tau\big),
\end{equation*}
where $\tau=\tau(x,t)$, $F=F(x,t,f)$, and
$$
\delta\sigma = \delta x\frac{\partial\sigma(x,t)}{\partial x}+\delta t\frac{\partial \sigma(x,t)}{\partial t},\qquad
\delta\tau=\delta x\frac{\partial\tau(x,t)}{\partial x}+\delta y\frac{\partial \tau(x,t)}{\partial t}.
$$
Now, by~\eqref{kylling} and~\eqref{broiler} we get
$$
\frac{\partial\sigma}{\partial x}=-\frac{\partial\tau}{\partial t},\qquad
\frac{\partial\sigma}{\partial t}=\frac{\partial\tau}{\partial x}.
$$
In view of \eqref{kota} and \eqref{sladdare} one  then checks that 
$J\circ D(\Phi\circ\Psi^{-1})=D(\Phi\circ\Psi^{-1})\circ J$. 
In the same way it follows that $J\circ D(\Psi\circ\Phi^{-1})=D(\Psi\circ\Phi^{-1})\circ J$. Hence,
$\Phi\circ\Psi^{-1}$ and $\Psi\circ\Phi^{-1}$ are sc-holomorphic and so $X$ is a complex M-polyfold. \qed
\end{example}

\bigskip

 In Example~\ref{ex:complexC2Mpolyfold}, the trivial direction of the Newlander-Nirenberg theorem was used. However, it is worth noting that while the Newlander-Nirenberg theorem holds for almost complex manifolds,  its non-trivial direction does not generalize to almost complex Banach 
manifolds (\hspace{1sp}\cite{patyi2000overline}). Therefore, assuming that the theorem holds in the M-polyfold case may be overly optimistic. However, the theorem may hold under certain additional assumptions, such as the finiteness of the dimension of $X$ as defined in Definition~\ref{def:dimofmpolyfold}.

\begin{conjecture}[Newlander-Nirenberg Theorem]
Let $(X,J)$ be an almost complex M-polyfold which is locally sc-diffeomorphic, at every point $p$, to some $\Gamma^{k(p)}_{n(p)}$  equipped with the complex structure defined in~\eqref{kota}, and~\eqref{sladdare}. Then, $X$ is a complex M-polyfold, in the sense of Definition \ref{def:integrable}, if and only if there is a sub-atlas of $X$ such that the transition functions between all charts (whenever they are defined) are sc-holomorphic.
\end{conjecture}

\end{document}